%% file: reg2cat.tex
\documentclass[11pt, oneside, article]{memoir}

\settrims{0pt}{0pt} 
\settypeblocksize{*}{34.2pc}{*} 
\setlrmargins{*}{*}{1} 
\setulmarginsandblock{1.1in}{1.1in}{*} 
\setheadfoot{\onelineskip}{2\onelineskip} 
\setheaderspaces{*}{1.5\onelineskip}{*} 
\checkandfixthelayout

\usepackage[centertags,sumlimits,intlimits,namelimits,reqno]{amsmath}
\usepackage{latexsym}
\usepackage{amsfonts,amsthm,amssymb}
\usepackage{mathtools}
\usepackage[cal=euler,scr=rsfso]{mathalfa}
\usepackage[boxslash]{stmaryrd}
\usepackage{newpxtext}
\usepackage{enumitem}
\usepackage{ifthen}
\usepackage{txfonts}
\usepackage[colorlinks,linkcolor=darkblue,citecolor=darkblue,urlcolor=darkblue,breaklinks=true, backref=true]{hyperref}
\usepackage{tikz}
\usepackage[capitalize]{cleveref}
\usepackage[backend=biber, backref, bibencoding=utf8, maxbibnames = 10, style = alphabetic]{biblatex}
\DeclareMathAlphabet{\mathpzc}{OT1}{pzc}{m}{it}
\usepackage{xstring}
\usepackage{graphicx}

\usepackage{multirow}
\usepackage[color=white, textsize=footnotesize]{todonotes}

\presetkeys%
    {todonotes}%
    {inline}{}
    
\DeclareFontFamily{U}{mathx}{\hyphenchar\font45}
\DeclareFontShape{U}{mathx}{m}{n}{
      <5> <6> <7> <8> <9> <10>
      <10.95> <12> <14.4> <17.28> <20.74> <24.88>
      mathx10
      }{}
\DeclareSymbolFont{mathx}{U}{mathx}{m}{n}
\DeclareFontSubstitution{U}{mathx}{m}{n}
\DeclareMathAccent{\widecheck}{0}{mathx}{"71}

	\allowdisplaybreaks
	
  \definecolor{darkblue}{rgb}{0,0,0.7}

  \usetikzlibrary{ 
  	cd,
  	math,
  	decorations.markings,
		decorations.pathreplacing,
  	positioning,
	 	circuits.logic.US,
 		arrows.meta,
  	shapes,
		shapes.geometric,
		shadows,
		shadings,
  	calc,
  	fit,
  	quotes,
  	intersections,
  }

  \input{tikz-stuff}

\DefineBibliographyStrings{english}{%
  backrefpage = {see page},
  backrefpages = {see pages},
}

  \setlist{noitemsep, nolistsep}
	\setlist[description]{leftmargin=0em, itemindent=2em}
		
\theoremstyle{plain}
\newtheorem*{theorem*}{Theorem} 
\newtheorem{theorem}{Theorem}[chapter] 
\newtheorem{proposition}[theorem]{Proposition}
\newtheorem*{proposition*}{Proposition}
\newtheorem{corollary}[theorem]{Corollary}
\newtheorem{lemma}[theorem]{Lemma}

\theoremstyle{definition}
\newtheorem{definition}[theorem]{Definition}

\newtheorem{notation}[theorem]{Notation}

\newtheorem*{axiom*}{Axiom}

\theoremstyle{remark}
\newtheorem{example}[theorem]{Example}
\newtheorem{remark}[theorem]{Remark}

\crefalias{chapter}{section}

\newcommand{\Set}[1]{\mathrm{#1}}
\newcommand{\ord}[1]{\underline{#1}}
\newcommand{\cat}[1]{\mathcal{#1}}
\newcommand{\ccat}[1]{\mathbb{#1}}
\newcommand{\Cat}[1]{{\mathsf{#1}}}
\newcommand{\CCat}[1]{\mathcal{\StrLeft{#1}{1}}\Cat{\StrGobbleLeft{#1}{1}}}
\newcommand{\poCat}[1]{\mathbb{\StrLeft{#1}{1}}\Cat{\StrGobbleLeft{#1}{1}}}

\newcommand{\ladj}{\Cat{LAdj}}
\newcommand{\radj}{\Cat{RAdj}}

\newcommand{\ladjto}{\to}
\newcommand{\radjto}{\to}

\DeclareMathOperator{\ob}{\Set{Ob}}
\DeclareMathOperator{\id}{id}

\DeclareMathOperator{\im}{im}
\DeclareMathOperator{\inc}{inc}
\DeclarePairedDelimiter{\pair}{\langle}{\rangle}

\DeclarePairedDelimiter{\copair}{[}{]}

\newcommand{\tn}[1]{\textnormal{#1}}
\newcommand{\op}{^{\tn{op}}}
\newcommand{\tp}{^{\dagger}}
\newcommand{\co}{^{\tn{co}}}

\newcommand{\inv}{^{-1}}
\newcommand{\tpow}[1]{^{\otimes #1}}

\newcommand{\name}[1]{#1^{\sqsubset}}

\newcommand{\rrgcat}{\CCat{RgCat}}
\newcommand{\rrlcat}{\CCat{RlPoCat}}
\newcommand{\rrel}{\poCat{Rel}}
\newcommand{\ccospan}{\CCat{Cspn}}
\newcommand{\sspan}{\CCat{Span}}

\newcommand{\finset}{\Cat{FinSet}}
\newcommand{\smset}{\Cat{Set}}

\newcommand{\cc}{\ccat{C}}
\newcommand{\dd}{\ccat{D}}

\newcommand{\rr}{\ccat{R}}
\newcommand{\nn}{\mathbb{N}}
\newcommand{\pp}{\mathbb{P}}
\newcommand{\qq}{\mathbb{Q}}
\newcommand{\ww}{\mathbb{W}}

\newcommand{\er}{\Set{ER}}

\newcommand{\cp}{\mathbin{\fatsemi}}
\newcommand{\cocolon}{:\!}

\newcommand{\To}[1]{\xrightarrow{#1}}
\newcommand{\Too}[1]{\To{\;\;#1\;\;}}
\newcommand{\from}{\leftarrow}
\newcommand{\From}[1]{\xleftarrow{#1}}

\newcommand{\surj}{\twoheadrightarrow}
\newcommand{\inj}{\rightarrowtail}

\renewcommand{\ss}{\subseteq}
\newcommand{\tto}{\Rightarrow}

\newcommand{\qqand}{\qquad\text{and}\qquad}

\newcommand{\adjphantom}[3][-.6pt]{\ar[#2, phantom, "#3" yshift=#1]}
\newcommand{\adj}[5][30pt]{
\begin{tikzcd}[ampersand replacement=\&, column sep=#1]
  #2\ar[r, shift left=5pt, "{#3}"]\adjphantom[0]{r}{\Rightarrow}\&
  #5\ar[l, shift left=5pt, "{#4}"]
\end{tikzcd}
}

\newcommand{\pb}[1][very near start]{\ar[dr, phantom, #1, "\lrcorner"]}

\begin{document}   

\title{Regular and relational categories:\\Revisiting `Cartesian bicategories I'}
\author{Brendan Fong \and David I.\ Spivak}
\date{\vspace{-.3in}}
  
\maketitle

\begin{abstract}
  Regular logic is the fragment of first order logic generated by $=$, $\top$, $\wedge$, and $\exists$. A key feature of this logic is that it is the minimal fragment required to express composition of binary relations; another is that it is the internal logic of regular categories. The link between these two facts is that in any regular category, one may construct a notion of binary relation using jointly-monic spans; this results in what is known as the bicategory of relations of the regular category. In this paper we provide a direct axiomatization of bicategories of relations, which we term \emph{relational po-categories}, reinterpreting the earlier work of Carboni and Walters along these lines. Our main contribution is an explicit proof that the 2-category of regular categories is equivalent to that of relational po-categories. Throughout, we emphasize the graphical nature of relational po-categories. 
\end{abstract}


\chapter{Introduction}

The goal in this paper is to abstract the properties of the bicategory of relations of a regular category.%
\footnote{We shall assume all our regular categories are well powered, i.e.\ that every object has a set of subobjects. This means we may talk of the category of relations without running into size issues.}

This is not a new task. Notably, Freyd and Scedrov's version of this story led to the notion of an allegory \cite{freyd1990categories}, while Carboni and Walters' version led to that of functionally complete `bicategories of relations' \cite{carboni1987cartesian}. Here we revisit the approach of Carboni and Walters, which emphasizes the monoidal structure on the bicategory of relations, as well as certain maps that coherently exist and form an algebraic structure on each object. In our terminology, we axiomatize the presence of all these cohering algebraic maps as a \emph{supply of wirings}.

Let us take as an example $\rrel$, the ur-bicategory of relations consisting of sets and binary relations. Note that it has a canonical symmetric monoidal structure, given by the cartesian product of sets, and that each homset has a canonical poset structure, given by inclusion of relations. Thus $\rrel$ forms a symmetric monoidal locally-posetal 2-category, or \emph{monoidal po-category} for short. 

Given any set $X$, for all $m,n \in \nn$ we have the diagonal binary relation $\{(x,\dots,x )\mid x \in X\} \subseteq X^m \times X^n$. We depict these with the following string diagrams, for various $m,n$:
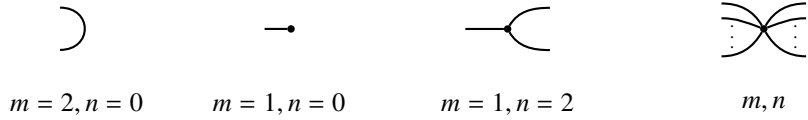
\[
	\begin{tikzpicture}[unoriented WD, font=\small]
	  \node (P1) {
    \begin{tikzpicture}[inner WD]
  	  \draw (0,0) to[out=0, in=0, looseness=2] (0,-2);
   	\end{tikzpicture}	
	  };
	  \node[right=7 of P1] (P2) {
  	\begin{tikzpicture}[inner WD]
      \node[link] (dot) {};
  	  \draw (dot) to +(-10pt, 0);
  	\end{tikzpicture}
	  };
	  \node[right=7 of P2] (P3) {
  	\begin{tikzpicture}[inner WD]
      \node[link] (delta) {};
      \draw (delta) -- +(-2,0);
      \draw (delta) to[out=60, in=180] +(2,1);
      \draw (delta) to[out=-60, in=180] +(2,-1);
  	\end{tikzpicture}
	  };
	  \node[right=7 of P3] (P4) {
  	\begin{tikzpicture}[inner WD]
      \node[link] (dot) {};
      \node at ($(dot)+(-1.5,0)$) {$\vdots$};
      \node at ($(dot)+(1.5,0)$) {$\vdots$};
      \draw (dot) to[out=120, in=0] +(-2,1.2);
      \draw (dot) to[out=160, in=0] +(-2,.5);
      \draw (dot) to[out=-120, in=0] +(-2,-1.3);
      \draw (dot) to[out=60, in=180] +(2,1.2);
      \draw (dot) to[out=20, in=180] +(2,.5);
      \draw (dot) to[out=-60, in=180] +(2,-1.3);
  	\end{tikzpicture}
	  };
    \node[below=1 of P1] (L1) {$m=2,n=0$};
    \node at (L1-|P2) {$m=1,n=0$};
    \node at (L1-|P3) {$m=1,n=2$};
    \node at (L1-|P4) {$m,n$};
\end{tikzpicture}
\]
These diagrams behave like wires in the sense that composition is about interconnection; for example:
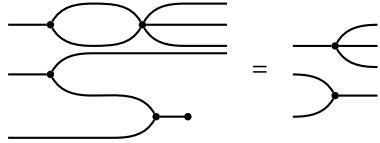
\[
	\begin{tikzpicture}[unoriented WD, font=\small]
	  \node (P1) {
	  \begin{tikzpicture}[inner WD]
      \node[link] (dot) {};
      \node[link,below=2 of dot] (dot2) {};
      \node[link,right=4 of dot] (dot3) {};
      \node[link] at ($(dot2)+(5,-2)$) (dot4) {};
  	  \draw (dot) to +(-2, 0);
      \draw (dot) to[out=60, in=180] +(2,1) to[out=0,in=120] (dot3); 
      \draw (dot) to[out=-60, in=180] +(2,-1) to[out=0,in=-120] (dot3); 
  	  \draw (dot3) to +(4, 0) coordinate (right);
      \draw (dot3) to[out=60, in=180] +(2,1) -- +(2,0);
      \draw (dot3) to[out=-60, in=180] +(2,-1) -- +(2,0);
  	  \draw (dot2) to +(-2, 0) coordinate (left);
      \draw (dot2) to[out=60, in=180] +(2,1) coordinate (hi) to (hi-|right); 
      \draw (dot2) to[out=-60, in=180] +(2,-1) to[out=0,in=120] (dot4); 
      \draw (dot4) to[out=-120, in=0] +(-2,-1) coordinate (bye) to (bye-|left) ;
      \draw (dot4) to +(1.5,0) node[link] {};
   	\end{tikzpicture}	
	  };
	  \node[right=2 of P1] (P2) {
  	\begin{tikzpicture}[inner WD]
      \node[link] (dot) {};
      \node[link,below=2 of dot] (dot2) {};
  	  \draw (dot) to +(-2, 0);
      \draw (dot) to[out=60, in=180] +(2,1);
  	  \draw (dot) to +(2, 0);
      \draw (dot) to[out=-60, in=180] +(2,-1);
  	  \draw (dot2) to +(2, 0);
      \draw (dot2) to[out=120, in=0] +(-2,1);
      \draw (dot2) to[out=-120, in=0] +(-2,-1);
  	\end{tikzpicture}
	  };
    \node at ($(P1.east)!.5!(P2.west)$) {$=$};
  \end{tikzpicture}
\]
Write $\ccospan$ for the bicategory of cospans of finite sets. More abstractly to equip an object $r$ in a po-category $\cc$ with a wiring structure is to provide a monoidal functor $\ccospan\co \to \cc$ sending $1$ to $r$. A po-category \emph{supplies wirings} if every object is equipped with a wiring, in a way compatible with the monoidal product.

With these wirings as syntax, we capture logical properties of relations using equations and inequalities. For example every binary relation $R \subseteq X \times Y$ obeys
\[
  \{x \mid \exists y. (x,y) \in R\} \subseteq X
  \quad \mbox{and} \quad
  \{(x,y,y) \mid (x,y) \in R\} \subseteq \{ (x,y_1,y_2) \mid (x,y_1) \in R, (x,y_2) \in R\}.
\]
In terms of wirings, we draw this:
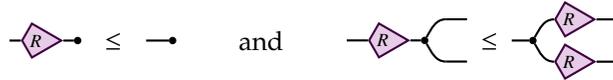
\[
	\begin{tikzpicture}[unoriented WD, font=\small]
	  \node (P1) {
	  \begin{tikzpicture}[inner WD]
	    \node[oshellr] (f) {$R$};
	    \node[link, right=.5 of f] (dot) {};
	    \draw (f) -- +(-10pt, 0);
      \draw (f) to (dot);
   	\end{tikzpicture}	
	  };
	  \node[right=2 of P1] (P2) {
  	\begin{tikzpicture}[inner WD]
      \node[link] (dot) {};
  	  \draw (dot) to +(-10pt, 0);
  	\end{tikzpicture}
	  };
	  \node at ($(P1.east)!.5!(P2.west)$) {$\leq$};
  	\node[right=7 of P2] (P3) {
  	 \begin{tikzpicture}[inner WD]
    	\coordinate (f1);
  		\coordinate[below=2 of f1] (f2);
  		\node[link] at ($(f1.west)!.5!(f2.west)+(-1,0)$) (dot) {};
    	\node[oshellr, left=.5 of dot] (f0) {$R$};
      \draw (f0.west) to +(-1,0);
      \draw (f0) to (dot);
  		\draw (f1.west) to[out=180, in=60] (dot);
  		\draw (f2.west) to[out=180, in=-60] (dot);
  		\draw (f1.east) -- +(1,0);
  		\draw (f2.east) -- +(1,0);
    \end{tikzpicture}
    };
   	\node[right=1 of P3] (P4) {
    \begin{tikzpicture}[inner WD]
    	\node[oshellr] (f1) {$R$};
  		\node[oshellr, below=.5 of f1] (f2) {$R$};
  		\draw (f1.east) -- +(1,0);
  		\draw (f2.east) -- +(1,0);
  		\node[link] at ($(f1.west)!.5!(f2.west)+(-1,0)$) (dot) {};
  		\draw (f1.west) to[out=180, in=60] (dot);
  		\draw (f2.west) to[out=180, in=-60] (dot);
  		\draw (dot) to +(-1,0);
    \end{tikzpicture}
     };
	  \node at ($(P2.east)!.5!(P3.west)$) {and};
	  \node at ($(P3.east)!.5!(P4.west)$) {$\leq$};
\end{tikzpicture}
\]
If these two inequalities hold for all morphisms $f$ in a monoidal po-category $\cc$, we say that the induced supply of comonoids is lax homomorphic, and that $\cc$ is a \emph{prerelational po-category}. In a prerelational po-category, we may express the property of being an adjoint or a monomorphism equationally. 

It is furthermore the case in $\rrel$ that every 1-morphism factors into a right adjoint followed by a left adjoint, satisfying a certain joint-monicness condition; i.e.\ it has a \emph{tabulation}. If $\cc$ is prerelational and has tabulations, we call it a \emph{relational po-category}; these are our main objects of study.

It is well known that the construction of the bicategory of relations $\rrel(\cat{R})$ does not forget any data of the regular category $\cat{R}$; the latter can be recovered internally as the category of left adjoints in $\rrel(\cat{R})$. Indeed, writing $\rrgcat$ for the 2-category of regular categories and $\rrlcat$ for the 2-category of relational po-categories, our main theorem is simply the following:

\begin{theorem*}
There is an equivalence of 2-categories
$
\begin{tikzcd}[column sep=30pt]
	\rrgcat\ar[r, shift left, "\rrel"]&
	\rrlcat\ar[l, shift left, "\ladj"]
\end{tikzcd}
$.
\end{theorem*} 

Moreover, we explicitly give a 2-equivalence of these categories that descends to a 1-equivalence on the underlying 1-categories. As one might expect, the functor $\rrel\colon \rrgcat \to \rrlcat$ takes a regular category and constructs its category of relations, while its inverse $\ladj\colon \rrlcat \to \rrgcat$ takes a relational po-category and returns its category of left adjoints.

The two equivalent structures are quite different on the surface: one is a 1-category with finite limits and pullback stable image factorizations; the other is a po-category that supplies wires and has tabulations. To prove the equivalence, Carboni and Walters went through a good deal of calculations, even though writing fairly tersely. The supply of wirings provides us with a graphical syntax with which to clarify many of those computations. As a byproduct, we will see that our computations in a relational po-category involve diagrams that are rather network-like, sensitive only to an indexing of boundary ports, and not to an exact allocation of domain and codomain. More formally, they take place in a hypergraph category \cite{fong2019hypergraph}.

Another way to understand the goal of reaxiomatization, is that we seek to lay out a minimal structure that allows us to interpret regular logic, the fragment of first order logic generated by $=$, $\top$, $\wedge$, and $\exists$. Here the wires allow us to express equality, the lax homomorphic structure gives us $\top$ and $\wedge$, and the tabulations give existential quantification.

\subsection*{Related work}
As we have mentioned, we are not the first to provide an axiomatization of relations, and we are largely revisiting the main results of Carboni and Walters' seminal paper `Cartesian bicategories I', leaning heavily on their insights. What we add here is: 
\begin{itemize}[nolistsep, noitemsep]
  \item an explicit 2-categorical equivalence, 
  \item more invariant definitions, made possible by a supply-oriented perspective,
  \item new, graphical proofs.
\end{itemize} 
That said, the major purpose of this paper is expository; we believe Carboni and Walters' work has been underappreciated. Walters in particular was vocal about the beauty of the graphical language of bicategories of relations, but their original 1987 paper was typeset in a time before the graphical tools we use became available. We hope this tribute to their ideas improves the availability of their insights to a wider audience, particularly those coming in from applications of category theory to networked systems and graphical languages. 

On this note, while the previous work of Freyd, Scedrov, Carboni, Walters and others shows a long history of interest in characterizing of categories of relations, the impetus for this work has been the more recent application of these ideas to network-style graphical languages, and the presence of implicit relational structure. The first examples include the work of Baez and Erbele \cite{baez2015categories}, and Bonchi, Sobocinski, and Zanasi \cite{Bonchi.Sobocinski.Zanasi:2017a}, on the signal flow graphs used in control theory and electrical engineering, which provide syntax for manipulating relations in the category of finite dimesional vector space. Recent work focussing on the principles behind these graphical languages includes that on relational theories by Bonchi, Pavlovic, and Sobocinski \cite{bonchi2017functorial}, and our companion paper \cite{fong2018graphical} on graphical regular logic.

\subsection*{Outline}

In \cref{chap.background} we review locally-posetal categories, which we call po-categories, and symmetric monoidal structures for them. In particular we will introduce po-props. 

In \cref{chap.reg_cats_rels} we review the notion of regular category and its associated po-category of relations. In \cref{chap.reg2cats}, we axiomatize the  resulting po-category as one that \emph{supplies wirings} and \emph{has tabulations}; we call such categories \emph{relational}. To do so, we need to define the po-prop $\ww$ of wirings as well as define the notions of supply and tabulation. In \cref{chap.relations_functor}, we prove that a regular category's po-category of relations is indeed relational and moreover that this construction is 2-functorial. 

In \cref{chap.ladj_functor}, we provide a 2-functor going back, from relational po-categories to regular categories. To do so requires an exploration of relational po-categories. We do so by using a graphical language that arises from the supply of wirings $\ww$. In \cref{sec.prove_equiv}, we prove our main theorem: we have a 2-equivalence of 2-categories, that descends to a 1-equivalence of 1-categories.

Finally, in \cref{chap.carboni_walters} we discuss how our definitions are equivalent to their corresponding notions in Carboni and Walters.

\subsection*{Acknowledgements}
Thanks to Christina Vasilakopoulou for pointing out an error in some of our previous work, that led to this re-examination of our starting assumptions.  We acknowledge support from AFOSR grants FA9550-17-1-0058 and FA9550-19-1-0113.

\subsection*{Notation}
\begin{itemize}[nolistsep, noitemsep]
\item Whenever we speak of 2-categories, 2-functors, or 2-natural transformations, we mean them in the strict sense. 
\item We often denote compositions of morphisms $f\colon x\to y$ and $g\colon y\to z$ in diagrammatic order, as $f\cp g\colon x\to z$.
\item Given a natural number $n\in\nn$, we write $\ord{n}\coloneqq\{1,2,\ldots,n\}$, e.g.\ $\ord{0}=\varnothing$.
\item Named 1-categories are denoted in san-serif, e.g.\ $\finset$; named 2-categories are denoted with a calligraphic first letter, e.g.\ $\rrgcat$. Generic po-categories are denoted in blackboard bold, e.g.\ $\cc$.
\end{itemize}

\chapter{Background}\label{chap.background}
Mostly to set our notation and terminology, we review the basic theory of symmetric monoidal locally-posetal 2-categories, or \emph{monoidal po-categories} for short.

\section{Po-categories and adjunctions}

We will use the following notion throughout the paper. 

\begin{definition}[Po-category]
A \emph{po-category} is a locally-posetal category $\cc$, i.e.\ it is an ordinary 1-category for which the set $\cc(c,c')$ of morphisms between any two objects has been equipped with a partial order, and for which composition is monotonic.

A \emph{po-functor} $F\colon\cc\to\dd$ between po-categories is an ordinary 1-functor for which the function $F(c,c')\colon\cc(c,c')\to\dd(Fc,Fc')$ is monotonic for any two objects $c,c'\in\cc$.

A \emph{natural transformation} $\alpha$ between po-functors $F,G\colon\cc\to\dd$ is an ordinary natural transformation $\alpha\colon F\tto G$ between the underlying 1-functors.

Given natural transformations $\alpha,\beta\colon F\tto G$, we write $\alpha\leq\beta$
\[
\begin{tikzcd}[column sep=90pt]
  \cc
  	\ar[r, bend left=25pt, "F", ""' {pos=.4, name=F1}, ""' {pos=.6, name=F2}]
		\ar[r, bend right=25pt, "G"', "" {pos=.4, name=G1}, "" {pos=.6, name=G2}]&
  \dd
  	\ar[from=F1, to=G1, Rightarrow, bend right, "\alpha"', "" name=al]
		\ar[from=F2, to=G2, Rightarrow, bend left, "\,\beta", ""' name=be]
		\ar[from=al, to=be, phantom, "\leq"]
\end{tikzcd}
\]
 if there is an inequality $\alpha_c\leq\beta_c$ in $\dd(Fc,Gc)$ for each object $c\in\cc$.

For any po-categories $\cc,\dd$, we denote by $[\cc,\dd]$ the po-category of po-functors and natural transformations; we call it the \emph{po-category of po-functors from $\cc$ to $\dd$}.
\end{definition}


\begin{definition}[Adjunction]\label{def.adjunction_in_a_2cat}
Let $\cc$ be a 2-category. A \emph{adjunction} in $\cc$ consists of a pair of objects $c,d\in\cc$, a pair of 1-morphisms $L\colon c\to d$ and $R\colon d\to c$, and a pair of 2-morphisms $\eta\colon
d\tto (L\cp R)$ and $\epsilon\colon (R\cp L)\tto c$, such that a pair of equations hold:
\begin{equation}\label{eqn.adjunction}
  \id_L =
  \begin{aligned}
    \begin{tikzcd}[row sep=4pt, column sep=large]
      c\ar[dr, "L"]\ar[dd, equal]\\
      \ar[r, phantom, pos=.3, "\overset{\eta}{\Longrightarrow}" above=-6pt]&
      d\ar[dl, "R" description]\ar[dd, equal]\\
      c\ar[dr, "L"']\ar[r, phantom, pos=.7, "\underset{\epsilon}{\Longrightarrow}" below=-5pt]&~\\&
      d
    \end{tikzcd}
  \end{aligned}
  \qquad\qqand\qquad
  \id_R =
  \begin{aligned}
    \begin{tikzcd}[row sep=4pt, column sep=large]
      &
      d\ar[dl, "R"']\ar[dd, equal]\\
      c\ar[dd, equal]\ar[dr, "L" description]\ar[r, phantom, pos=.7, "\overset{\epsilon}{\Longrightarrow}" above=-6pt]&
      ~\\
      \ar[r, phantom, pos=.3, "\underset{\eta}{\Longrightarrow}" below=-7pt]&
      d\ar[dl, "R"]\\
      c
    \end{tikzcd}
  \end{aligned}
\end{equation}
Noting that both $\eta$ and $\epsilon$ always point in the direction of the left
adjoint $L$, we write
\[\adj{c}{L}{R}{d}\] 
to denote an adjunction, where the 2-arrow points in the direction of the left
adjoint.
\end{definition}

\begin{lemma}\label{lemma.funs_ladjs}
2-functors preserve left adjoints and right adjoints.
\end{lemma}

Note that in any 2-category, any two right adjoints to a 1-morphism $L$ are isomorphic, so in a po-category, a 1-morphism has \emph{at most one} right adjoint.

\begin{proposition}[Adjoint natural transformations]\label{prop.adj_in_pocat}
Let $\cc$ and $\dd$ be po-categories, and suppose given a pair of morphisms in the po-category $[\cc,\dd]$ as follows:
\[
\begin{tikzcd}
  F\ar[r, shift left, "\lambda"]&
  G.\ar[l, shift left, "\rho"]
\end{tikzcd}\]
Then $\lambda$ is left adjoint to $\rho$ iff the component 1-morphism $\lambda_c\colon F(c)\to G(c)$ is left adjoint to $\rho_c\colon G(c)\to F(c)$ in $\dd$ for each $c\in\cc$.
\end{proposition}
\begin{proof}
The forward direction is true even for 2-categories that aren't locally posetal; the backwards direction holds by the uniqueness of 2-morphisms in a po-category.
\end{proof}

\section{Symmetric monoidal po-categories}

In this paper, we will have a great deal of use for symmetric monoidal po-categories $\cc$. Strictly speaking, these are 3-categories, but they are ``petite" in two directions: they are locally posetal, and they have only one object. It is conceptually simpler to think of $\cc$ as a 1-category with extra structure: hom-sets are equipped with an order, and there is a symmetric monoidal operation on objects and morphisms. 

\begin{definition}
A \emph{symmetric monoidal structure} on a po-category $\cc$ consists of a symmetric monoidal structure on its underlying 1-category, such that $(f_1\otimes g_1)\leq(f_2\otimes g_2)$ whenever $f_1\leq f_2$ and $g_1\leq g_2$.

A \emph{strong symmetric monoidal po-functor} is simply a po-functor whose underlying 1-functor is equipped with a strong symmetric monoidal structure.
\end{definition}

If $F\colon(\cat{C},I,\otimes)\to(\cat{D},J,\odot)$ is a strong monoidal functor, we refer to the coherence maps $\varphi\colon J\to F(I)$ and $\varphi\colon F(c)\odot F(c')\to F(c\otimes c')$ for each $c,c'$ as the \emph{strongators} for $F$.

If $m,n\in\nn$ are natural numbers, and $c\colon \ord{m}\times \ord{n}\to\cc$ is a family of objects in $\cc$, we have a natural isomorphism
\begin{equation}\label{eqn.symmetry}
\sigma\colon
\bigotimes_{i\in\ord{m}}\bigotimes_{j\in\ord{n}}c(i,j)\Too{\cong}
\bigotimes_{j\in\ord{n}}\bigotimes_{i\in\ord{m}}c(i,j).
\end{equation}
We refer to $\sigma$ as the \emph{symmetry} isomorphism, though note that it involves associators and unitors too, not just the symmetric braiding. We will be interested in two particular cases of the isomorphism \cref{eqn.symmetry}, namely for $m=2$ and $m=0$ and any $n\in\nn$:
\[\sigma\colon c_1\tpow{n}\otimes c_2\tpow{n}\Too{\cong}(c_1\otimes c_2)\tpow{n}
\qqand
\sigma\colon I\Too{\cong} I\tpow{n}.
\]

\begin{definition}[Po-prop]\label{def.props}
A \emph{po-prop} is a symmetric strict monoidal po-category $\pp$ whose monoid of objects is equal to $(\nn,0,+)$.

A \emph{functor} $F\colon\pp\to\qq$ between po-props is an identity-on-objects monoidal functor.
\end{definition}


\begin{proposition}\label{prop.ladj_smc}
For any po-category $\cc$, the left-adjoints in $\cc$ form a wide subcategory $\ladj(\cc)\ss\cc$, and there is an isomorphism of categories $\ladj(\cc)\cong\radj(\cc)\op$. If $\cc$ is symmetric monoidal then so are $\ladj(\cc)$ and $\ladj(\cc)\cong\radj(\cc)\op$.

Moreover, any (symmetric monoidal) po-functor $F\colon\cc\to\cc'$ induces a (symmetric monoidal) functor $\ladj(F)\colon\ladj(\cc)\to\ladj(\cc')$, and any left adjoint natural transformation $\alpha\colon F\to G$ induces a natural transformation $\ladj(F)\to\ladj(G)$.
\end{proposition}
\begin{proof}
It is easy to check that every identity is a left adjoint and a right adjoint. Left adjoints are closed under composition, with the right adjoint of their composite given by the composite of their right adjoints (paste diagrams in \cref{eqn.adjunction}). Thus we have a 1-category $\ladj(\cc)$, a wide inclusion $\ladj(\cc)\ss\cc$, and an isomorphism of categories $\ladj(\cc)\cong\radj(\cc)\op$. The monoidal product of left adjoints is a left adjoint, with its adjoint the monoidal product of the individual adjoints.

Given a (symmetric monoidal) po-functor $F$ one obtains a (symmetric monoidal) functor $\ladj(F)$ by \cref{lemma.funs_ladjs}. By \cref{prop.adj_in_pocat}, the components $\alpha_c\colon F(c)\to G(c)$ of a left adjoint natural transformation $\alpha\colon F\to G$ are left adjoints, thus defining the components of a natural transformation $\ladj(F)\to\ladj(G)$, and required the naturality squares are inherited from $\alpha$.
\end{proof}

\begin{proposition} \label{prop.cartesian_nat_trans}
Suppose that $\cat{C},\cat{D}$ are cartesian monoidal categories and $F,G\colon\cat{C}\to\cat{D}$ are strong monoidal functors. Then any natural transformation $\alpha\colon F\to G$ is monoidal.
\end{proposition}
\begin{proof}
This follows from the universal properties of terminal objects and products.
\end{proof}

\chapter{Regular categories and their relations}\label{chap.reg_cats_rels}

The main characters of this paper are two identical twins: regular categories and and relational po-categories. In this section \ we discuss the former. In \cref{sec.regcats} we review the axioms for regular categories $\cat{R}$, and then in \cref{sec.pocat_rels} we describe the po-category $\rrel(\cat{R})$ of relations in $\cat{R}$. In \cref{chap.reg2cats} we will axiomatize those po-categories that arise in this way as \emph{relational po-categories}, our second main character.

\section{The 2-category of regular categories}\label{sec.regcats}

The definition of regular category collects a few notions that make sense in any category.

\begin{definition}[Extremal epimorphism]
Let $\cat{C}$ be a category and $e\colon r\to s$ a morphism. We say $e$ an \emph{extremal epimorphism} if it has the property that whenever $e=f \cp m$ is a factorization of $e$, and $m$ is a monomorphism, then in fact $m$ is an isomorphism.
\end{definition}

\begin{definition}[Image factorization]
Given a morphism $f\colon r\to s$, we say that a factorization $f=e\cp m$ is an \emph{image factorization} if $m$ is a monomorphism and $e$ is an extremal epimorphism.
\end{definition}

\begin{definition}[$\rrgcat$] \label{def.regcat}
A category $\cat{R}$ is called \emph{regular} if:
  \begin{enumerate}[label=(\roman*)]
	\item it has all finite limits;
	\item every morphism in $\cat{R}$ has an image factorization; and
	\item extremal epimorphisms are pullback stable. That is, for any pullback diagram
	\[
	\begin{tikzcd}
		r'\ar[r, "e'"]\ar[d]\pb&
		s'\ar[d]\\
		r\ar[r, ->>, "e"']&
		s
	\end{tikzcd}
	\]
	if $e$ is an extremal epimorphism then so is $e'$.
\end{enumerate}
A functor $F\colon\cat{R}\to\cat{R'}$ is called \emph{regular} if it preserves finite limits and extremal epimorphisms.

We denote by $\rrgcat$ the 2-category whose objects are regular categories, whose morphisms are regular functors, and whose 2-morphisms are natural transformations.\footnote{Note that no additional property is needed here: it is automatically the case that these natural transformations are monoidal with respect to the cartesian monoidal structure; see \cref{prop.cartesian_nat_trans}.}
\end{definition}

\begin{example}
The categories $\finset$, $\finset\op$, are both regular. In fact any topos is regular, as is its opposite. Any category monadic over $\smset$ is regular. 
\end{example}

\begin{remark}
  There is an issue of size relevant to the constructions in this paper. A category is called \emph{well powered} if every object has a \emph{set} of subobjects. We shall assume all our regular categories are well powered.
\end{remark}

\begin{remark}\label{rem.standard_abuse}
Given two objects $r,s$ in a regular category $\cat{R}$, we follow the standard category-theoretic style and allow ourselves to write $r\times s$ to denote their product. Implicitly, this---together with a choice $1$ of terminal object---means we are \emph{choosing} a monoidal structure $(1,\times)$ on $\cat{R}$; and this of course involves the axiom of choice. In order to use the $\times$ symbol, it would be slightly more honest to define a regular category to be a cartesian monoidal category with equalizers and image factorizations---i.e.\ to make the monoidal structure explicit---but we decided to go along with the standard definition and the standard use of choice. One could avoid choice using the approach of \cite{makkai1996avoiding}.
\end{remark}

\begin{proposition}
In a regular category, if $f=e\cp m$ is an image factorization and $f=g_1\cp g_2$ is any other factorization, then there exists a dotted lift making the following diagram commute:
\[
\begin{tikzcd}
	r\ar[r, "g_1"]\ar[d, ->>, "e"']&
	x\ar[d, >->, "g_2"]\\
	\im(f)\ar[ur, dashed, "h" description]\ar[r, >->, "m"']&
	s
\end{tikzcd}
\]
\end{proposition}
\begin{proof}
The pullback of $g_2$ along $m$ is monic and hence an isomorphism since $e$ is extremal. Use its inverse and a projection map to construct $h$.
\end{proof}

\section{The po-category of relations}\label{sec.pocat_rels}

An important property---arguably the key property---of regular categories is that they support a well-behaved notion of relation. A relation between objects $r$ and $s$ in a regular category $\cat{R}$ is a subobject of the product $r \times s$. Equivalently, we may speak in terms of jointly-monic spans.

A jointly-monic span $r \to s$ is a pair of morphisms $r \From{f} x \To{g} s$ such that the pairing $\pair{f,g}\colon x \to r \times s$ is a monomorphism. A morphism from one jointly-monic spans $r \From{f} x \To{g} s$ to another $r \From{f'} x' \To{g'} s$ is a morphism $k\colon x \to x'$ in $\cat{R}$ such that
\begin{equation}\label{eqn.morphisms_order}
  \begin{tikzcd}[row sep=0pt]
    & x \ar[dl, "f"'] \ar[dr, "g"] \ar[dd, pos=.3, "k"]\\
    r & & s \\
    & x' \ar[ul, "f'"] \ar[ur, "g'"'] 
  \end{tikzcd}
\end{equation}
commutes. A \emph{relation} $r \to s$ is an isomorphism class of jointly-monic spans. We will frequently abuse terminology and refer to the representative spans themselves as relations.

It is easy to check that relations $r \to s$ form a poset, where $\pair{f,g} \le \pair{f',g'}$ if there exists a morphism from $\pair{f,g}$ to $\pair{f',g'}$.
Moreover, relations can be composed. Given $r\From{f_1} x\To{g_1} s$ and $s\From{f_2}y\To{g_2}t$, we define their composite by first taking the pullback $x\times_{s}y$, and then taking the image of its induced map to $r\times t$.
The result is independent of the choice of the span representing $x$. Finally, the identity morphism on $r$ is the identity span $r\From{\id}r\To{\id}r$.

\begin{definition} \label{def.relations_construction}
  Let $\cat{R}$ be a regular category. Define its \emph{po-category of relations} $\rrel(\cat{R})$ to be the symmetric monoidal po-category with $\ob(\cat{R})$ as objects, relations as morphisms, composition and identity as above, and monoidal structure given by categorical products in $\cat{R}$; see \cref{rem.standard_abuse}.
\end{definition}


It is well-known that the above description indeed defines a symmetric monoidal po-category---this follows from the universal properties of products, pullbacks, and images, see for example \cite[Theorem 2.8.4]{Borceux:1994b}. Our next goal is to axiomatize the sorts of po-categories that arise in this way.

\chapter{Relational po-categories}\label{chap.reg2cats}

In this section we define the second of our main characters: relational po-categories; these are po-categories for which a certain algebraic structure is coherently supplied to each object, and where every morphism can be factored in a certain way. In order to state a precise definition, we first define a po-prop $\ww$ in \cref{sec.prop_ww} and the notion of supply in \cref{sec.supply}, at which point we can define pre-relational and relational po-categories in \cref{sec.def_reg2cat}. It turns out that even prerelational po-categories have a vast amount of useful structure, which we will study in \cref{chap.relations_functor}.
\section{The po-prop $\ww$ for wiring}\label{sec.prop_ww}

Consider the symmetric monoidal 2-category $(\ccospan\co,\varnothing,+)$, i.e.\ the 2-dual of cospans between finite sets. The full subcategory spanned by finite ordinals $\ord{n}$ is a skeleton, and we denote its local poset reflection by $\ww$. It is a po-prop, and we describe its hom-categories explicitly in \cref{prop.ww_explicit}.


\begin{definition}[$\ww$]\label{def.wwiring}
We refer to the po-prop $\ww$ defined above as the \emph{po-prop for wiring}.
\end{definition}

\begin{remark}
One can motivate the definition of $\ww$ as follows. A regular category $\cat{R}$ has finite products, and thus each object $r$ is equipped with morphisms $\epsilon_r\colon r\to 1$ and $\delta_r\colon r\to r\otimes c$. The category $\finset\op$ is the free finite product category, and---in a sense that we will soon make precise---the theory of comonoids.

We will see that in the bicategory of relations $\rrel(\cat{R})$, morphisms coming from $\cat{R}$ are precisely the left adjoints. It was shown in \cite[Theorem A.2]{Hermida:2000a} that the span construction freely adds right adjoints, subject to the condition that pullbacks in $\cat{R}$ are sent to Beck-Chevalley squares. Since all of our categories are po-categories, we are using the local posetal reflection $\ww$ of $\sspan(\finset\op)$.
\end{remark}

$\ww$ is almost the prop of equivalence relations---see \cite{coya2017corelations}---but without the ``extra'' law, which would equate cospans $0\to 0\from 0$ and $0\to 1\from 0$.

\begin{proposition}\label{prop.ww_explicit}
The hom-posets of $\ww$ admit the following explicit description:
\[
  \ww(m,n)\cong
  \begin{cases}
  	\{0\leq 1\}\op&\tn{ if }m=n=0\\
		\er\op(m+n)&\tn{ if }m+n\geq 1
  \end{cases}
\]
where $\{0\leq 1\}$ is the poset of booleans, and $\er(p)$ is the poset of equivalence relations on the set $\ord{p}$, ordered by inclusion.
\end{proposition}
\begin{proof}
For any $m,n$, may identify $\ww(m,n)\op$ with the poset reflection of $\ccospan(\ord{m},\ord{n})$. This is the coslice category $\ord{m+n}/\finset$, consisting of sets $S$ equipped with a function $\ord{m+n}\to S$. If $m+n=0$ this poset reflection is that of $\finset$, namely $\{\varnothing\leq\{1\}\}$; otherwise it may be identified with the poset of equivalence relations on $\ord{m+n}$. Indeed, every function $\ord{m+n}\to S$ factors as an epic followed by a monic, and every monic out of a nonempty set has a retraction.
\end{proof}

The morphisms of $\ww$ can be generated by four morphisms, for which we have special notation and iconography:
\setlength{\belowrulesep}{0pt}
\setlength{\tabcolsep}{8pt}
\renewcommand{\arraystretch}{1.5}
\begin{equation}\label{eqn.generating_wires}
	\begin{tabular}{|c|c|c|}\toprule[1pt]
    Morphism in $\ww$ & Corresponding cospan & Icon \\\toprule[1pt]
    $\epsilon\colon 1\to 0$& 
    	$1\to 1\from 0$&
				\begin{tikzpicture}[WD]
        	\node[link] (epsilon) {};
    			\draw (epsilon) to +(-.8,0);
    		\end{tikzpicture}\\\hline
		$\delta\colon 1\to 2$&
			$1\to 1\from 2$&
				\begin{tikzpicture}[WD]
        	\node[link] (delta) {};
        	\draw (delta) -- +(-1,0);
        	\draw (delta) to[out=60, in=180] +(1,.5);
        	\draw (delta) to[out=-60, in=180] +(1,-.5);
				\end{tikzpicture}\\\hline
		$\eta\colon 0\to 1$&
			$0\to 1\from 1$&
				\begin{tikzpicture}[WD]
        	\node[link] (eta) {};
  	  		\draw (eta) to +(.8,0);
    		\end{tikzpicture}\\\hline
		$\mu\colon 2\to 1$&
			$2\to 1\from 1$&
				\begin{tikzpicture}[WD]
        	\node[link] (mu) {};
        	\draw (mu) -- +(1,0);
        	\draw (mu) to[out=120, in=0] +(-1,.5);
        	\draw (mu) to[out=-120, in=0] +(-1,-.5);
				\end{tikzpicture}\\\toprule[1pt]
	\end{tabular}
\end{equation}
\setlength{\tabcolsep}{3pt}
These generators satisfy various equations and inequalities:
\begin{equation}\label{eqn.equations_wires}
\renewcommand{\arraystretch}{2}
\begin{array}{rl<{\qquad}rl<{\qquad}rl}
	\begin{tikzpicture}[WD]
  	\node[link] (mu) {};
  	\draw (mu) -- +(-1,0);
  	\draw (mu) to[out=60, in=180] +(.5,.5) coordinate (out1);
  	\draw (mu) to[out=-60, in=180] +(.5,-.5) coordinate (out2);
  	\coordinate (end) at ($(out1)+(1,0)$);
  	\draw (out1) to[out=0, in=180] (out2-|end);
  	\draw (out2) to[out=0, in=180] (out1-|end);	
	\end{tikzpicture}
&\;\raisebox{2pt}{=}\quad
	\begin{tikzpicture}[WD]
  	\node[link] (mu) {};
  	\draw (mu) -- +(-1,0);
  	\draw (mu) to[out=60, in=180] +(1,.5);
  	\draw (mu) to[out=-60, in=180] +(1,-.5);
	\end{tikzpicture}
&
  \begin{tikzpicture}[WD]
  	\node[link] (mu) {};
  	\draw (mu) -- +(-1,0);
  	\draw (mu) to[out=60, in=180] +(.5,.5) node[link] {};
  	\draw (mu) to[out=-60, in=180] +(1,-.5) -- +(.5,0);
	\end{tikzpicture}
&\;\raisebox{2pt}{=}\quad
  \begin{tikzpicture}[WD,baseline=-5pt]
		\draw (0,0) -- (2,0);
	\end{tikzpicture}
&
	\begin{tikzpicture}[WD]
  	\node[link] (mu) {};
  	\draw (mu) -- +(-1,0);
  	\draw (mu) to[out=60, in=180] +(.5,.5) node[link] (mu2) {};
  	\draw (mu2) to[out=60, in=180] +(1,.5) coordinate (end);
  	\draw (mu2) to[out=-60, in=180] +(1,-.5);
  	\draw (mu) to[out=-60, in=180] +(1,-.5) coordinate (h);
		\draw (h) -- (h-|end);
	\end{tikzpicture}
&\;\raisebox{2pt}{=}\quad
	\begin{tikzpicture}[WD]
  	\node[link] (mu) {};
  	\draw (mu) -- +(-1,0);
  	\draw (mu) to[out=-60, in=180] +(.5,-.5) node[link] (mu2) {};
  	\draw (mu2) to[out=60, in=180] +(1,.5) coordinate (end);
  	\draw (mu2) to[out=-60, in=180] +(1,-.5);
  	\draw (mu) to[out=60, in=180] +(1,.5) coordinate (h);
		\draw (h) -- (h-|end);
	\end{tikzpicture}
\\
	\begin{tikzpicture}[WD]
  	\node[link] (mu) {};
  	\draw (mu) -- +(1,0);
  	\draw (mu) to[out=120, in=0] +(-.5,.5) coordinate (out1);
  	\draw (mu) to[out=-120, in=0] +(-.5,-.5) coordinate (out2);
  	\coordinate (end) at ($(out1)+(-1,0)$);
  	\draw (out1) to[out=180, in=0] (out2-|end);
  	\draw (out2) to[out=180, in=0] (out1-|end);	
	\end{tikzpicture}
&\;\raisebox{2pt}{=}\quad
	\begin{tikzpicture}[WD]
  	\node[link] (mu) {};
  	\draw (mu) -- +(1,0);
  	\draw (mu) to[out=120, in=0] +(-1,.5);
  	\draw (mu) to[out=-120, in=0] +(-1,-.5);
	\end{tikzpicture}
&
  \begin{tikzpicture}[WD]
  	\node[link] (mu) {};
  	\draw (mu) -- +(1,0);
  	\draw (mu) to[out=120, in=0] +(-.5,.5) node[link, anchor=center] {};
  	\draw (mu) to[out=-120, in=0] +(-1,-.5) -- +(-.5,0);
 	\end{tikzpicture}
&\;\raisebox{2pt}{=}\quad
	\begin{tikzpicture}[WD,baseline=-5pt]
		\draw (0,0) -- (2,0);
	\end{tikzpicture}
&
	\begin{tikzpicture}[WD]
  	\node[link] (mu) {};
  	\draw (mu) -- +(1,0);
  	\draw (mu) to[out=120, in=0] +(-.5,.5) node[link, anchor=center] (mu2) {};
  	\draw (mu2) to[out=120, in=0] +(-1,.5) coordinate (end);
  	\draw (mu2) to[out=-120, in=0] +(-1,-.5);
  	\draw (mu) to[out=-120, in=0] +(-1,-.5) coordinate (h);
		\draw (h) -- (h-|end);
	\end{tikzpicture}
&\;\raisebox{2pt}{=}\quad
	\begin{tikzpicture}[WD]
  	\node[link] (mu) {};
  	\draw (mu) -- +(1,0);
  	\draw (mu) to[out=-120, in=0] +(-.5,-.5) node[link, anchor=center] (mu2) {};
  	\draw (mu2) to[out=120, in=0] +(-1,.5) coordinate (end);
  	\draw (mu2) to[out=-120, in=0] +(-1,-.5);
  	\draw (mu) to[out=120, in=0] +(-1,.5) coordinate (h);
		\draw (h) -- (h-|end);
	\end{tikzpicture}
\\
	\begin{tikzpicture}[WD]
		\node[link] (delta) {};
    \draw (delta) -- +(-.75,0);
    \draw (delta) to[out=60, in=180] +(.5,.5) coordinate (out1);
    \draw (delta) to[out=-60, in=180] +(.5,-.5) coordinate (out2);
    \node[link, right=1 of delta] (mu) {};
		\draw (mu) -- +(.75,0);
    \draw (mu) to[out=120, in=0] +(-.5,.5) -- (out1);
    \draw (mu) to[out=-120, in=0] +(-.5,-.5) -- (out2);
	\end{tikzpicture}
&\;\raisebox{2pt}{=}\quad
	\begin{tikzpicture}[WD,baseline=-5pt]
 		\draw (0,0) -- (1.5,0);	
	\end{tikzpicture}
&
\multicolumn{4}{c}{$
  \begin{tikzpicture}[WD]
    \coordinate (htop);
    \coordinate[below=.7 of htop] (hmid);
    \coordinate[below=.7 of hmid] (hbot);
 		\node[link, left=.5] (dotL) at ($(htop)!.5!(hmid)$) {};
    \node[link, right=.5] (dotR) at ($(hbot)!.5!(hmid)$) {};
 		\draw (dotL) -- +(-1,0) coordinate (l);
    \draw (dotR) -- +(1,0) coordinate (r);
 		\draw (dotL) to[out=60, in=180] (htop);
    \draw (dotL) to[out=-60, in=180] (hmid);
 		\draw (hmid) to[out=0, in=120] (dotR);
    \draw (hbot) to[out=0, in=-120] (dotR);
 		\draw (htop) -- (htop-|r);
  	\draw (hbot) -- (hbot-|l);
  \end{tikzpicture}
\;\raisebox{2pt}{=}\quad
	\begin{tikzpicture}[WD]
    \node[link] (delta) {};
   	\draw (delta) -- +(-.5,0) coordinate (end);
    \draw (delta) to[out=60, in=180] +(1,.5);
   	\draw (delta) to[out=-60, in=180] +(1,-.5);
    \node[link, left=1 of delta] (mu) {};
   	\draw (mu) -- (end);
    \draw (mu) to[out=120, in=0] +(-1,.5);
   	\draw (mu) to[out=-120, in=0] +(-1,-.5);
	\end{tikzpicture}
\;\raisebox{2pt}{=}\quad
  \begin{tikzpicture}[WD]
  	\coordinate (htop);
    \coordinate[below=.7 of htop] (hmid);
    \coordinate[below=.7 of hmid] (hbot);
    \node[link, left=.5] (dotL) at ($(hbot)!.5!(hmid)$) {};
    \node[link, right=.5] (dotR) at ($(htop)!.5!(hmid)$) {};
    \draw (dotL) -- +(-1,0) coordinate (l);
    \draw (dotR) -- +(1,0) coordinate (r);
    \draw (dotL) to[out=-60, in=180] (hbot);
    \draw (dotL) to[out=60, in=180] (hmid);
    \draw (hmid) to[out=0, in=-120] (dotR);
    \draw (htop) to[out=0, in=120] (dotR);
    \draw (htop) -- (htop-|l);
    \draw (hbot) -- (hbot-|r);
  \end{tikzpicture}
$}
\\
			\begin{tikzpicture}[WD,baseline=-5pt]
  			\draw (0,0) -- (1.5,0);
			\end{tikzpicture}			
&\;\raisebox{2pt}{$\leq$}\quad
			\begin{tikzpicture}[WD,baseline=-5pt]
      	\node[link] (epsilon) {};
      	\draw (epsilon) to +(-.8,0);
      	\node[link, right=.5 of epsilon] (eta) {};
      	\draw (eta) to +(.8,0);
			\end{tikzpicture}
&
			\begin{tikzpicture}[WD,baseline=-5pt]
      	\node[link] (epsilon) {};
      	\draw (epsilon) to +(-.8,0);
      	\node[link, left=1 of epsilon] (eta) {};
      	\draw (eta) to +(.8,0);
			\end{tikzpicture}
&\;\raisebox{2pt}{$\leq$}\quad
\id_0
&
			\begin{tikzpicture}[WD]
      	\node[link] (delta) {};
      	\draw (delta) -- +(-.5,0) coordinate (end);
      	\draw (delta) to[out=60, in=180] +(1,.5);
      	\draw (delta) to[out=-60, in=180] +(1,-.5);
      	\node[link, left=1 of delta] (mu) {};
      	\draw (mu) -- (end);
      	\draw (mu) to[out=120, in=0] +(-1,.5);
      	\draw (mu) to[out=-120, in=0] +(-1,-.5);
			\end{tikzpicture}
&\;\raisebox{2pt}{$\leq$}\quad
			\begin{tikzpicture}[WD]
   			\draw (0, 0) -- (1.5, 0);	
   			\draw (0, .75) -- (1.5, .75);		
			\end{tikzpicture}
\end{array}
\end{equation}

We refer to $\eta\cp\delta$ and $\mu\cp\epsilon$ as the \emph{cup} and \emph{cap}; they are depicted as follows:
\begin{equation}\label{eqn.cap_cup}
\begin{array}{cccc}
	\begin{tikzpicture}[inner WD]
		\draw (.5, .5) -- (0,.5) to[out=180, in=180] (0, -.5) -- (.5, -.5);
	\end{tikzpicture}	
	&\;\raisebox{2pt}{$\coloneqq$}\quad
	\begin{tikzpicture}[WD]
  	\node[link] (mu) {};
		\node[link, left=.3 of mu] (eta) {};
  	\draw (mu) -- (eta);
  	\draw (mu) to[out=60, in=180] +(.5,.5) -- +(.25,0);
  	\draw (mu) to[out=-60, in=180] +(.5,-.5) -- +(.25,0);
	\end{tikzpicture}
	&\qqand
	\begin{tikzpicture}[inner WD]
		\draw (-.5, .5) -- (0,.5) to[out=0, in=0] (0, -.5) -- (-.5, -.5);
	\end{tikzpicture}	
	&\;\raisebox{2pt}{$\coloneqq$}\quad
	\begin{tikzpicture}[WD]
  	\node[link] (mu) {};
		\node[link, right=.3 of mu] (eta) {};
  	\draw (mu) -- (eta);
  	\draw (mu) to[out=120, in=0] +(-.5,.5) -- +(-.25,0);
  	\draw (mu) to[out=-120, in=0] +(-.5,-.5) -- +(-.25,0);
	\end{tikzpicture}	
\end{array}
\end{equation}

\renewcommand{\arraystretch}{1}
The equations in the first and second lines of \cref{eqn.equations_wires} are called \emph{commutativity, unitality, associativity} for comonoids and monoids, respectively. The equations in the next line are called the \emph{special law} and the \emph{frobenius law}. We refer to the inequalities in the last line as the \emph{adjunction inequalities}, because they show up as the unit and counit of adjunctions, as we see next in \cref{prop.left_adj_in_ww}.

\begin{proposition}\label{prop.left_adj_in_ww}
With notation as in \eqref{eqn.generating_wires}, there are adjunctions
\[
  \adj{1}{\epsilon}{\eta}{0}
  \qqand
  \adj{1}{\delta}{\mu}{2}.
\]
\end{proposition}
\begin{proof}
The inequalities $\id_1\leq(\epsilon\cp\eta)$,\quad $(\eta\cp\epsilon)\leq\id_0$,\quad $\id_2\leq(\delta\cp\mu)$, and the equation $\id_1=(\mu\cp\delta)$ are all shown in \eqref{eqn.equations_wires}, which itself is proved via computations in $\ccospan\co$.
\end{proof}

\begin{remark}\label{rem.surprising_inequality}
The perhaps surprising inequality $(\mu\cp\delta)\leq\id_1$, i.e.\ the one not coming from adjointness, is in fact derivable from the rest of the structure:
  \[
    \begin{tikzpicture}
      \node (P1) {
        \begin{tikzpicture}[WD]
          \node[link] (delta) {};
          \draw (delta) -- +(-.75,0);
          \draw (delta) to[out=60, in=180] +(.5,.5) coordinate (out1);
          \draw (delta) to[out=-60, in=180] +(.5,-.5) coordinate (out2);
          \node[link, right=1 of delta] (mu) {};
          \draw (mu) -- +(.75,0);
          \draw (mu) to[out=120, in=0] +(-.5,.5) -- (out1);
          \draw (mu) to[out=-120, in=0] +(-.5,-.5) -- (out2);
        \end{tikzpicture}
      };
      \node (P2) [right=1 of P1] {
        \begin{tikzpicture}[WD]
          \node[link] (delta) {};
          \draw (delta) -- +(-.75,0);
          \draw (delta) to[out=60, in=180] +(.5,.5) coordinate (out1);
          \draw (delta) to[out=-60, in=180] +(.5,-.5) coordinate (out2);
          \node[link, right=1 of delta] (mu) {};
          \draw (mu) to[out=120, in=0] +(-.5,.5) -- (out1);
          \draw (mu) to[out=-120, in=0] +(-.5,-.5) -- (out2);
          \draw (mu) -- +(.75,0) node[link] (delta2) {};
          \draw (delta2) to[out=60, in=180] +(.5,.5) -- +(.5,0);
          \draw (delta2) to[out=-60, in=180] +(.5,-.5) node[link] {};
        \end{tikzpicture}
      };
      \node (P3) [right=1 of P2] {
        \begin{tikzpicture}[WD]
          \node[link] (delta) {};
          \draw (delta) -- +(-.75,0);
          \draw (delta) to[out=60, in=180] +(.5,.5) -- +(.5,0);
          \draw (delta) to[out=-60, in=180] +(.5,-.5) node[link] {};
        \end{tikzpicture}
      };
      \node (P4) [right=1 of P3] {
        \begin{tikzpicture}[WD]
          \draw (0,0) -- (1.5,0);
        \end{tikzpicture}
      };
      \node at ($(P1.east)!.5!(P2.west)$) {$=$};
      \node at ($(P2.east)!.5!(P3.west)$) {$\leq$};
      \node at ($(P3.east)!.5!(P4.west)$) {$=$};
    \end{tikzpicture}
  \]
\end{remark}

\section{Supply}\label{sec.supply}

It often happens that every object in a symmetric monoidal category $\cat{C}$ is equipped with the same sort of algebraic structure---say coming from a prop $\pp$---with the property that these algebraic structures are compatible with the monoidal structure. In \cite{fong2019supplying}, the authors refer to this situation by saying that $\cat{C}$ \emph{supplies} $\pp$. For our purposes we need to slightly generalize this story, from props to po-props and from symmetric monoidal categories $\cat{C}$ to symmetric monoidal po-categories $\cc$.

\begin{definition}[Supply]\label{def.supply}
Let $\pp$ be a po-prop and $\cc$ a symmetric monoidal po-category. A \emph{supply of $\pp$ in $\cc$} consists of a strong monoidal po-functor $s_c\colon\pp\to\cc$ for each object $c\in\cc$, such that
\begin{enumerate}[label=(\roman*)]
	\item $s_c(m)=c\tpow{m}$ for each $m\in\nn$, 
	\item the strongator $c\tpow{m}\otimes c\tpow{n}\to c\tpow{(m+n)}$ is equal to the associator for each $m,n\in\nn$, and
	\item the following diagrams commute for every $c,d\in\cc$ and $\mu\colon m\to n$ in $\pp$:
\begin{equation}\label{eqn.supply_commute_tensors}
\begin{tikzcd}[column sep=55pt]
	c\tpow{m}\otimes d\tpow{m}\ar[r, "s_c(\mu)\otimes s_d(\mu)"]\ar[d, "\sigma"']&
	c\tpow{n}\otimes d\tpow{n}\ar[d, "\sigma"]\\
	(c\otimes d)\tpow{m}\ar[r, "s_{c\otimes d}(\mu)"']&
	(c\otimes d)\tpow{n}
\end{tikzcd}
\hspace{.7in}
\begin{tikzcd}
	I\ar[r, equal]\ar[d, "\sigma"']&
	I\ar[d, "\sigma"]\\
	I\tpow{m}\ar[r, "s_I(\mu)"']&
	I\tpow{n}
\end{tikzcd}
\end{equation}
where the $\sigma$'s are the symmetry isomorphisms from \cref{eqn.symmetry}. 
\end{enumerate}
We often denote the morphism $s_c(\mu)$ in $\cc$ simply by $\mu_c\colon c\tpow{m}\to c\tpow{n}$ for typographical reasons; i.e.\ we elide explicit mention of $s$.

We further say that $f\colon c\to d$ in $\cc$ is a \emph{lax $s$-homomorphism} (resp.\ \emph{oplax $s$-homomorphism}) if, for each $\mu\colon m\to n$ in the prop $\pp$, there is a 2-morphism as shown in the left-hand (resp.\ right-hand) diagram:
\begin{equation}\label{eqn.nat_means_homo}
\begin{tikzcd}
	c\tpow{m}\ar[r, "\mu_c"]\ar[d, "f\tpow{m}"']&
	c\tpow{n}\ar[d, "f\tpow{n}"]\\
	d\tpow{m}\ar[r, "\mu_d"']&
	d\tpow{n}\ar[ul, phantom, "\leq"]
\end{tikzcd}
\hspace{2cm}
\begin{tikzcd}
	c\tpow{m}\ar[r, "\mu_c"]\ar[d, "f\tpow{m}"']&
	c\tpow{n}\ar[d, "f\tpow{n}"]\\
	d\tpow{m}\ar[r, "\mu_d"']&
	d\tpow{n}\ar[ul, phantom, "\geq"]
\end{tikzcd}
\end{equation}
Since $\cc$ is locally posetal, if $f$ is both a lax and an oplax $s$-homomorphism, then these diagrams commute and we simply say $f$ is an \emph{$s$-homomorphism}.

We say that $\cc$ \emph{(lax-/oplax-) homomorphically supplies $\pp$} if every morphism $f$ in $\cc$ is a (lax/oplax) $s$-homomorphism.
\end{definition}

An important example of supply is that categories with finite products are those that homomorphically supply commutative comonoids. This is the main theorem of \cite{fox1976coalgebras}.

\begin{proposition}\label{prop.fox}
A category $\cat{C}$ has finite products iff it can be equipped with a homomorphic supply of commutative comonoids. If $\cat{C}$ and $\cat{D}$ have finite products, a functor $\cat{C}\to\cat{D}$ preserves them iff it preserves the supply of comonoids.
\end{proposition}



\cref{cor.change_of_supply} below was essentially proven in \cite[Proposition~3.18]{fong2019supplying}; the change from props to po-props makes no difference in this context.

\begin{proposition}[Change of supply]\label{cor.change_of_supply}
Let $G\colon\pp\to\qq$ be a po-prop functor. For any supply $s$ of $\qq$ in $\cc$, we have a supply $(G\cp s)$ of $\pp$ in $\cc$.
\end{proposition}

\begin{definition}[Preservation of supply]\label{def.preserve_supply}
Let $\pp$ be a po-prop, $\cc$ and $\dd$ symmetric monoidal po-categories, and suppose $s$ is a supply of $\pp$ in $\cc$ and $t$ is a supply of $\pp$ in $\dd$. We say that a strong symmetric monoidal po-functor $(F,\varphi)\colon\cc\to\dd$ \emph{preserves the supply} if the strongators $\varphi$ provide an isomorphism $t_{Fc}\cong (s_c\cp F)$ of po-functors $\pp\to\dd$ for each $c\in\cc$.
\end{definition}

Unpacking, a strong monoidal po-functor $(F,\varphi)$ preserves the supply iff the following diagram commutes for each morphism $\mu\colon m\to n$ in $\pp$ and object $c\in\cc$:
\begin{equation}\label{eqn.unpack_preserve_supply}
	\begin{tikzcd}[column sep=large]
  	F(c)\tpow{m}\ar[r, "\mu^{\;}_{F(c)}"]\ar[d, "\varphi"', "\cong"]&
  	F(c)\tpow{n}\ar[d, "\varphi", "\cong"']\\
  	F(c\tpow{m})\ar[r, "F(\mu_c)"']&
  	F(c\tpow{n})
  \end{tikzcd}
\end{equation}


\section{Definition of relational categories}\label{sec.def_reg2cat}

The goal of this section is to axiomatize po-categories that arise as the relations of a regular categories; we call them  \emph{relational po-categories} and define them in \cref{def.reg2cat}. Most of the structure is already present in the following definition.

\begin{definition}[Prerelational po-category] \label{def.prerel_pocat}
  A symmetric monoidal po-category $\cc$ is \emph{prerelational} if it supplies wiring $\ww$, such that the supply of comonoids is lax homomorphic.
\end{definition}

\begin{remark}
  The lax comonoid homomorphism condition requires that for every $f\colon c\to d$ in $\cc$, one has 
  $f\cp\delta_d\leq\delta_c\cp (f\otimes f)$ and $
  f\cp\epsilon_d\leq\epsilon_c$. In graphical notation:
\begin{equation}\label{eqn.lax_hom_prerel}
  \begin{aligned}
	\begin{tikzpicture}[unoriented WD, font=\small]
	  \node (P1) {
	  \begin{tikzpicture}[inner WD]
	    \node[oshellr] (f) {$f$};
	    \node[link, right=.5 of f] (dot) {};
	    \draw (f) -- +(-10pt, 0);
      \draw (f) to (dot);
   	\end{tikzpicture}	
	  };
	  \node[right=2 of P1] (P2) {
  	\begin{tikzpicture}[inner WD]
      \node[link] (dot) {};
  	  \draw (dot) to +(-10pt, 0);
  	\end{tikzpicture}
	  };
	  \node at ($(P1.east)!.5!(P2.west)$) {$\leq$};
  	\node[right=7 of P2] (P3) {
  	 \begin{tikzpicture}[inner WD]
    	\coordinate (f1);
  		\coordinate[below=2 of f1] (f2);
  		\node[link] at ($(f1.west)!.5!(f2.west)+(-1,0)$) (dot) {};
    	\node[oshellr, left=.5 of dot] (f0) {$f$};
      \draw (f0.west) to +(-1,0);
      \draw (f0) to (dot);
  		\draw (f1.west) to[out=180, in=60] (dot);
  		\draw (f2.west) to[out=180, in=-60] (dot);
  		\draw (f1.east) -- +(1,0);
  		\draw (f2.east) -- +(1,0);
    \end{tikzpicture}
    };
   	\node[right=1 of P3] (P4) {
    \begin{tikzpicture}[inner WD]
    	\node[oshellr] (f1) {$f$};
  		\node[oshellr, below=.5 of f1] (f2) {$f$};
  		\draw (f1.east) -- +(1,0);
  		\draw (f2.east) -- +(1,0);
  		\node[link] at ($(f1.west)!.5!(f2.west)+(-1,0)$) (dot) {};
  		\draw (f1.west) to[out=180, in=60] (dot);
  		\draw (f2.west) to[out=180, in=-60] (dot);
  		\draw (dot) to +(-1,0);
    \end{tikzpicture}
     };
	  \node at ($(P2.east)!.5!(P3.west)$) {and};
	  \node at ($(P3.east)!.5!(P4.west)$) {$\leq$};
\end{tikzpicture}
\end{aligned}
\end{equation}
\end{remark}

\begin{lemma} \label{lemma.lax_comons_oplax_mons}
  Suppose $\cc$ supplies wiring $\ww$. Then the supply of comonoids is lax homomorphic iff the supply of monoids is oplax homomorphic.
\end{lemma}
\begin{proof}
  Note that given $l$ left adjoint to $r$, we always have that $a \le (l \cp b)$ iff $(r \cp a) \le b$ and that $(a \cp l)\le b$ iff $a \le (b \cp r)$. The result now follows from \cref{prop.left_adj_in_ww}.
\end{proof}

It follows from this that every hom-poset $\cc(c,d)$ in a prerelational po-category is a meet-semilattice. Though this is an interesting fact, we will not need to use it; we thus save the proof until \cref{prop.meetsl}.

\begin{notation}\label{rem.frob_means_freedom}
Recall that we refer to the generating morphisms of $\ww$ as $(\eta,\mu,\epsilon,\delta)$; see \cref{eqn.generating_wires}. If $\cc$ supplies wiring $\ww$ then the equations in \cref{eqn.equations_wires} imply that $(\eta,\mu,\epsilon,\delta)$ form a \emph{special commutative frobenius monoid} structure on every object, and these structures are compatible with $\cc$'s monoidal structure. This is exactly the definition of $\cc$, or more precisely its underlying 1-category, being a \emph{hypergraph category} (see \cite{fong2019hypergraph}).
  
In particular, one obtains a self-dual compact closed structure on $\cc$ using the cup and cap; see \cref{eqn.cap_cup}. Thus to every 1-morphism $f\colon c \to d$ we may associate its \emph{transpose} $f\tp\colon d \to c$. In string diagrams, we depict the transpose of
  $
  \begin{tikzpicture}[inner WD, baseline=(f.-20)]
  	\node[oshellr] (f) {$f$};
		\draw (f.west) -- +(-.5,0);
		\draw (f.east) -- +(.5,0);
	\end{tikzpicture}
	$
	by its 180-degree rotation,
\[
  \begin{tikzpicture}[inner WD, baseline=(f.-20)]
  	\node[oshelll] (f) {$f$};
		\draw (f.west) -- +(-.5,0);
		\draw (f.east) -- +(.5,0);
	\end{tikzpicture}
	\quad\coloneqq\quad
  \begin{tikzpicture}[inner WD, baseline=(f.-20)]
  	\node[oshellr] (f) {$f$};
		\draw (f.west) -- 
			++(-.5, 0) to[out=180, in=180]
			++(0, 1.25) --
			++(4.5, 0)
			;
		\draw (f.east) -- 
			++(.5,0) to[out=0, in=0]
			++(0, -1.25) --
			++(-4.5, 0)
			;
	\end{tikzpicture}
\]
  
	The additional structures and axioms of a hypergraph category allow us to split wires in various ways, but such that ``connectivity is all that matters'' when interpreting string diagrams. Thus we may use network or circuit-like diagrams as in \cite{fong2019hypergraph} to represent 1-morphisms in $\cc$. This allows us to be rather informal when depicting morphisms as diagrams---in particular how a given connection is constructed from frobenius maps---because all formalizations of it result in the same morphism. For example, the following diagrams all represent the same morphism:
  
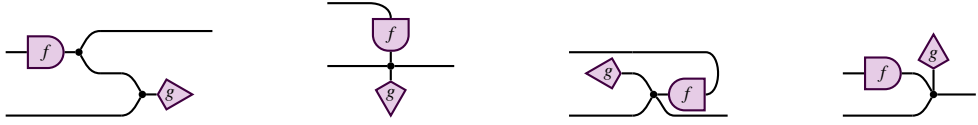
\[
    \begin{tikzpicture}[inner WD]
    	\coordinate (f1);
  		\coordinate[below=2 of f1] (f2);
  		\node[link] at ($(f1)!.5!(f2)+(-1,0)$) (dot) {};
  		\coordinate[right=1 of f2] (g1);
  		\coordinate[below=2 of g1] (g2);
  		\node[link] at ($(g1)!.5!(g2)+(1,0)$) (dot2) {};
    	\node[funcr, left=.5 of dot] (f0) {$f$};
    	\node[oshellr, right=.5 of dot2,inner sep=1pt] (g0) {$g$};
      \draw (f0.west) to +(-1,0);
      \draw (f0) to (dot);
  		\draw (f1.west) to[out=180, in=60] (dot);
  		\draw (f2.west) to[out=180, in=-60] (dot);
      \draw (f1) -- (f1-|g0) -- +(2,0);
  		\draw (f2) -- (g1);
  		\draw (g1) to[out=0, in=120] (dot2);
  		\draw (g2) to[out=0, in=-120] (dot2);
      \draw (g0) to (dot2);
      \draw (g2) to (g2-|f0.west) -- +(-1,0);
    \end{tikzpicture}
    \hspace{1.5cm}
    \begin{tikzpicture}[inner WD]
  		\node[link] (dot) {};
    	\node[funcd, above=.5 of dot] (f0) {$f$};
    	\node[oshelld, below=.5 of dot,inner sep=1pt] (g0) {$g$};
      \draw (f0) -- (dot);
      \draw (g0) -- (dot);
      \draw (dot) -- +(-3,0);
      \draw (dot) -- +(3,0);
  		\draw (f0) to[out=90, in=0] +(-1,1.5) -- +(-2,0);
    \end{tikzpicture}
    \hspace{1.5cm}
    \begin{tikzpicture}[inner WD]
  		\coordinate (g1);
  		\coordinate[below=2 of g1] (g2);
  		\node[link] at ($(g1)!.5!(g2)+(1,0)$) (dot) {};
    	\node[funcl, right=.5 of dot] (f0) {$f$};
    	\node[oshelll, left=.5 of g1,inner sep=1pt] (g0) {$g$};
      \draw (f0) -- (dot);
      \draw (g0) -- (g1);
  		\draw (g1) to[out=0, in=120] (dot);
  		\draw (g2) to[out=0, in=-120] (dot);
      \draw (g2) -- +(-3,0);
      \draw (dot) to[out=-45, in=180] +(1,-1) -- +(2.5,0);
      \draw (f0.east) to[out=0, in=0] +(0,2) -- ($(dot-|g1)+(0,2)$);
      \draw ($(dot-|g1)+(0,2)$) -- +(-3,0);
    \end{tikzpicture}
    \hspace{1.5cm}
    \begin{tikzpicture}[inner WD]
  		\coordinate (g1);
  		\coordinate[below=2 of g1] (g2);
  		\node[link] at ($(g1)!.5!(g2)+(1,0)$) (dot) {};
    	\node[funcr, left=.5 of g1] (f0) {$f$};
    	\node[oshellu, above=1 of dot,inner sep=1pt] (g0) {$g$};
      \draw (f0.west) -- +(-1,0);
      \draw (f0.east) -- (g1);
  		\draw (g1) to[out=0, in=120] (dot);
  		\draw (g2) to[out=0, in=-120] (dot);
      \draw (dot) -- +(2,0);
      \draw (g0) -- (dot);
      \draw (g2) -- (g2-|f0.west) -- +(-1,0);
    \end{tikzpicture}
  \]
  
  We refer the reader unfamiliar with this notation to \cite{fong2019hypergraph} for more details.
\end{notation}

To make a prerelational category relational, we need tabulations. The following definition is due to Freyd and Scedrov \cite{freyd1990categories}.

\begin{definition}[Tabulation] \label{def.tabulation}
Suppose $\cc$ supplies $\ww$ and let $f\colon r \tickar s$ be a morphism in $\cc$. A \emph{tabulation $(f_R,f_L)$ of $f$} is a factorization $r\Tickar{f_R}|f|\Tickar{f_L}s$ of $f$ where
	\begin{enumerate}[label=(\roman*)]
		\item $f_R\colon r\tickar |f|$ is a right adjoint;
		\item $f_L\colon |f|\tickar s$ is a left adjoint; and
		\item $\hat{f}\cp\hat{f}\tp=\id_{|f|}$, where $\hat{f}\coloneqq\delta_{|f|}\cp(f_L\otimes f_R\tp)$; in pictures
  \begin{equation}\label{eqn.tabulation}
  \begin{tikzpicture}[baseline=(P1)]
  	\node (P1) {
    \begin{tikzpicture}[inner WD, baseline=(dot1)]
    	\node[oshellr] (fl) {$f_L$};
  		\node[oshelll, right=of fl] (fl*) {$f_L$};
  		\node[oshellr, below=of fl] (fr!) {$f_R$};
  		\node[oshelll] at (fl*|-fr!) (fr) {$f_R$};
  		\draw (fl.east) -- (fl*.west);
  		\draw (fr!.east) -- (fr.west);
  		\node[link] at ($(fl.west)!.5!(fr!.west)+(-1,0)$) (dot1) {};
  		\node[link] at ($(fl*.east)!.5!(fr.east)+(1,0)$) (dot2) {};
  		\draw (fl.west) to[out=180, in=60] (dot1);
  		\draw (fr!.west) to[out=180, in=-60] (dot1);
  		\draw (dot1) to node[above, font=\scriptsize] {$|f|$} +(-2,0);
  		\draw (fl*.east) to[out=0, in=120] (dot2);
  		\draw (fr.east) to[out=0, in=-120] (dot2);
  		\draw (dot2) to node[above, font=\scriptsize] {$|f|$} +(2,0);
    \end{tikzpicture}
    };
    \node at ($(P1.east)+(.5,0)$) {$=$};
    \draw ($(P1.east)+(1,0)$) to node[above, font=\scriptsize] {$|f|$} +(1,0);
\end{tikzpicture}
\end{equation}
\end{enumerate}
We call $\hat{f}\colon|f|\to r\otimes s$ the \emph{span associated to $f$}.
\end{definition}

We will see in \cref{prop.rrel_tabulations} that a tabulation in $\rrel(\cat{R})$ is just a jointly monic span in the regular category $\cat{R}$. This holds in full generality; see \cref{prop.condition3}.

\begin{definition}[Relational po-category]\label{def.reg2cat}
A \emph{relational po-category} is a prerelational po-category $\rr$ in which every morphism has a tabulation.

We denote by $\rrlcat$ the po-category whose objects are relational po-categories $\rr$, whose 1-morphisms are strong monoidal po-functors $F\colon\rr\to\rr'$, and whose 2-morphisms are left adjoint natural transformations $\alpha\colon F\tto G$.
\end{definition}

\begin{remark}
We will eventually construct a 2-functor $\ladj\colon\rrlcat\to\rrgcat$ that sends every $\rr$ to its category of left adjoints. For $\ladj$ to be 2-functorial, the 2-morphisms in $\rrlcat$ need to be left adjoint transformations by \cref{prop.ladj_smc}.
\end{remark}

%

\begin{remark}
  A prerelational po-category is exactly what Carboni and Walters called a \emph{`bicategory of relations'} (quotation marks were an explicit part of the their terminology), and a relational po-category is exactly what they called a \emph{functionally complete `bicategory of relations'}. There are a few, ultimately immaterial differences in the definition; we will examine these in \cref{chap.carboni_walters}.
\end{remark}

\begin{example}\label{ex.ww_not_relational}
  Note that the po-category $\ww$ is prerelational, but not relational: the cospan $0\to 1\from 0$ in $\ww$ does not have a tabulation.
\end{example}

\begin{example}
  As we shall see in the next section, for any regular category $\cat{R}$, the po-category $\rrel(\cat{R})$ is relational.
\end{example}

\chapter{The relations 2-functor $\mathbb{R}\Cat{el}\colon\mathcal{R}\Cat{gCat}\to\mathcal{R}\Cat{lPoCat}$}\label{chap.relations_functor}

We have said that for any regular category $\cat{R}$, the po-category $\rrel(\cat{R})$ of relations is relational in the sense of \cref{def.reg2cat}. In this section we prove that fact, as well as the functoriality of the $\rrel$ construction.

\section{$\mathbb{R}\Cat{el}(\cat{R})$ is relational}

Given a category $\cat{R}$ with finite limits, write $\sspan(\cat{R})$ for the usual category of spans in $\cat{R}$. It has a symmetric monoidal structure given by the categorical product in $\cat{R}$. There is an identity-on-objects symmetric monoidal functors $\inc_0\colon\cat{R} \to \sspan(\cat{R})$ sending $f \colon r \to s$ to the span $r \From{\id_r} r \To{f} s$, or equivalently to its graph $\pair{\id_r,f}\colon r\to r\times s$.

If $\cat{R}$ is regular, there is a functor $i\colon \sspan(\cat{R}) \to \rrel(\cat{R})$ sending a span $r \From{f} x \To{g} s$ to the image factorization of $\pair{f,g}$. We denote the composite (identity-on-objects, monoidal) functor by $\inc\coloneqq (\inc_0\cp i)\colon\cat{R}\to\rrel(\cat{R})$; it again sends  $f\mapsto\pair{\id_r,f}$.

\begin{proposition}\label{prop.rr_supplied}
Let $\cat{R}$ be a regular category. The supply of comonoids in $\cat{R}$ extends to a supply of wirings in $\rrel(\cat{R})$.
\end{proposition}
\begin{proof}
Given an object $r \in \cat{R}$, the comonoid supplied to $r$ is given by the unique finite product preserving functor $\finset\op \to \cat{R}$. But every finite product preserving functor out of $\finset\op$ preserves all finite limits \cite[Lemma VIII.4.1]{MacLane.Moerdijk:1992a}, and thus this functor extends to a symmetric monoidal functor $\ccospan\co =\sspan(\finset\op) \to \sspan(\cat{R})$. Composing it with the functor $i\colon \sspan(\cat{R}) \to \rrel(\cat{R})$ induces a po-functor $s_r\colon \ww \to \rrel(\cat{R})$. The fact that these functors $s_r$ form a supply of $\ww$ in $\rrel(\cat{R})$ (i.e.\ the diagrams in \eqref{eqn.supply_commute_tensors} commute) follows from the fact that we started with a supply of comonoids ($\finset\op$), and every morphism in $\ccospan\co$ is an adjoint of a morphism in $\finset\op$; see \cref{prop.ladj_smc}.
\end{proof}

Thus $\rrel(\cat{R})$ is a prerelational po-category. To show that it is a relational po-category, it remains to prove that it has tabulations. 

\begin{lemma}\label{lemma.rr_lax_supply_comonoid}
Every morphism $f\colon r\tickar s$ in $\rrel(\cat{R})$ is a lax comonoid homomorphism.
\end{lemma}
\begin{proof}
We need to establish the inequalities $f\cp\epsilon_s\leq\epsilon_r$ and $f\cp\delta_s\leq\delta_r\cp(f\otimes f)$ (see \eqref{eqn.lax_hom_prerel}).
Taking $f$ to be a jointly monic span $r\from x\to s$, each inequality amounts to taking two pullbacks and comparing. We leave this calculation to the reader.
\end{proof}

\begin{lemma}\label{rem.rels_adjoint_composites}
  Let $f\colon r \to s$ in $\cat{R}$. In $\rrel(\cat{R})$, the graph $\pair{\id_r,f}$ of $f$ is left adjoint to the co-graph $\pair{f,\id_r}$ of $f$.
\end{lemma}
\begin{proof}
  This is straightforward; the unit of the adjunction is the canonical map given by the universal property of the pullback, while the counit is simply given by $f$ itself.
\end{proof}

\begin{remark}
  In fact, in \cref{lemma.fundamental} we shall see that graphs and cographs characterise left and right adjoints in $\rrel(\cat{R})$; we call this the \emph{fundamental lemma of regular categories}.
\end{remark}

\begin{proposition}\label{prop.rrel_tabulations}
If $\cat{R}$ is a regular category then $\rrel(\cat{R})$ has tabulations.
\end{proposition}
\begin{proof}
Let $r \From{f}x \To{g} s$ be a relation---that is, a jointly-monic span---in $\cat{R}$. Then
$\pair{f,g}=\pair{f,\id_x} \cp \pair{\id_x,g}$in $\rrel(\cat{R})$, so by \cref{rem.rels_adjoint_composites} every morphism in $\rrel(\cat{R})$ can be written as the composite of a right adjoint followed by a left adjoint. Thus we have established conditions (i), (ii) of \cref{def.tabulation}. It remains to show (iii), that $(\pair{f,\id_x}, \pair{\id_x,g})$ obeys \cref{eqn.tabulation}. But this is true because $\pair{f,g}\colon x\to r\times s$ is monic in $\cat{R}$ by assumption, and the pullback of a monic along itself is the identity.
\end{proof}

\begin{corollary}\label{cor.rrel_on_obs}
  Let $\cat{R}$ be a regular category. Then $(\rrel(\cat{R}),I,\otimes)$ is a relational po-category.
\end{corollary}
\begin{proof}
This follows from \cref{prop.rr_supplied,lemma.rr_lax_supply_comonoid,prop.rrel_tabulations}.
\end{proof}

\section{$\mathbb{R}\Cat{el}$ is functorial}

Our next goal is to establish that $\rrel$ extends to a 2-functor.

\begin{proposition}\label{lemma.ff_strong_mon_func}
  Let $F\colon \cat{R} \to \cat{R}'$ be a regular functor. Then we may define a strong monoidal po-functor $\rrel(F)\colon \rrel(\cat{R}) \to \rrel(\cat{R}')$ sending object $r$ to $Fr$, and relation $\pair{f,g}$ to $\pair{Ff,Fg}$. The strongators are given by the canonical isomorphisms between finite products.
\end{proposition}
\begin{proof}
   Recall that $\rrel(\cat{R})$ is defined to have the same objects as $\cat{R}$, so the proposed data is well typed, and $\rrel(F)$ obviously preserves the order \eqref{eqn.morphisms_order} on morphisms. It also preserves identity and composition---as is obvious from the definition (see \cref{def.relations_construction}) of these structures---because regular functors preserve diagonals, pullbacks, and image factorizations by \cref{def.regcat}. Finally, $\rrel(F)$ is strong monoidal because the strongators are given by universal properties. 
\end{proof}

\begin{proposition} \label{prop.rrel_on_2cells}
  Let $F,G\colon\cat{R}\to\cat{R}'$ be regular functors. A natural transformation $\alpha\colon F\tto G$ induces a left adjoint natural transformation $\rrel(\alpha)\colon\rrel(F)\to\rrel(G)$. 
\end{proposition}
\begin{proof}
For each $r\in\ob\cat{R}=\ob\rrel(\cat{R})$, the graph of the component $\alpha_r\colon F(r)\to G(r)$ is a left adjoint in $\rrel(\cat{R}')$ by \cref{rem.rels_adjoint_composites}; thus by \cref{prop.adj_in_pocat} we have defined the components of a left adjoint natural transformation $\rrel(\alpha)\colon\rrel(F)\to\rrel(G)$.
\end{proof}

\begin{theorem}\label{thm.rrel}
There is a 2-functor $\rrel\colon\rrgcat\to\rrlcat$.
\end{theorem}
\begin{proof}
We defined $\rrel$ on objects $\cat{R}$ in \cref{cor.rrel_on_obs}, on 1-morphisms in \cref{lemma.ff_strong_mon_func}, and on 2-morphism in \cref{prop.rrel_on_2cells}.
\end{proof}

\chapter{The left adjoints 2-functor $\ladj\colon\mathcal{R}\Cat{lPoCat}\to\mathcal{R}\Cat{gCat}$}\label{chap.ladj_functor}

We will show in \cref{thm.eequivalence} that the 2-functor $\rrel\colon\rrgcat\to\rrlcat$ from \cref{chap.relations_functor} is part of an equivalence of 2-categories. In this section, we provide a functor $\ladj$ in the reverse direction, which sends a relational category to its category of left adjoints. Our first goal is to better understand left adjoints in prerelational po-categories.

\section{Left adjoints in a prerelational po-category}\label{sec.leftadjs}

Recall the notion of adjunction in a 2-category $\cc$ from \cref{def.adjunction_in_a_2cat}. When $\cc$ is a po-category, the triangle equations \eqref{eqn.adjunction} automatically hold, so we can depict the situation that $L$ is left adjoint to $R$ as the following inequalities:
\begin{equation}\label{eqn.show_adjoints}
	    \begin{tikzpicture}[inner WD, baseline=-.5ex]
	      \draw (-15pt,0) to (15pt,0);
      \end{tikzpicture}	
      \overset{\text{unit}}{\le}
	    \begin{tikzpicture}[inner WD, baseline=-.5ex]
	      \node[oshellr] (f) {$L$};
	      \node[oshelll, right=.75 of f] (g) {$R$};
	      \draw (f.west) -- +(-.5, 0);
	      \draw (f) -- (g);
	      \draw (g.east) -- +(.5,0);
      \end{tikzpicture}	
\qqand
	    \begin{tikzpicture}[inner WD, baseline=-.5ex]
	      \node[oshelll] (f) {$L$};
	      \node[oshellr, right=.75 of f] (g) {$R$};
	      \draw (f.west) -- +(-.5, 0);
	      \draw (f) -- (g);
	      \draw (g.east) -- +(.5,0);
      \end{tikzpicture}	
      \overset{\text{counit}}{\le}
	    \begin{tikzpicture}[inner WD, baseline=-.5ex]
	      \draw (-15pt,0) to (15pt,0);
      \end{tikzpicture}	
\end{equation}
We will see in \cref{cor.leftadjs} that whenever $L$ is left adjoint to $R$, it is also the transpose $L\tp=R$. In this section we consider variants of the inequalities in \cref{eqn.show_adjoints}.

\begin{proposition} \label{prop.adj_vs_hom}
  Let $f\colon r \to s$ be a morphism in a prerelational po-category. For each row in the table below, the homomorphism property implies the corresponding adjointness property.
  \setlength{\belowrulesep}{0pt}
  \begin{center}
\setlength{\tabcolsep}{18pt}
\renewcommand{\arraystretch}{1.5}
    \begin{tabular}{|c|c|}
      \toprule[1pt]
      Homomorphism property & Adjointness property \\
      \toprule[1pt]
      comultiplication homomorphism \textsc{(h1)} & deterministic \textsc{(a1)} 
      \\
    \begin{tikzpicture}[inner WD,baseline=-2.2ex]
    	\coordinate (f1);
  		\coordinate[below=2 of f1] (f2);
  		\node[link] at ($(f1.west)!.5!(f2.west)+(-1,0)$) (dot) {};
    	\node[oshellr, left=.5 of dot] (f0) {$f$};
      \draw (f0.west) to +(-.5,0);
      \draw (f0) to (dot);
  		\draw (f1.west) to[out=180, in=60] (dot);
  		\draw (f2.west) to[out=180, in=-60] (dot);
  		\draw (f1.east) -- +(.5,0);
  		\draw (f2.east) -- +(.5,0);
    \end{tikzpicture} 
    $=$
    \begin{tikzpicture}[inner WD, baseline=-2.2ex]
    	\node[oshellr] (f1) {$f$};
  		\node[oshellr, below=.5 of f1] (f2) {$f$};
  		\draw (f1.east) -- +(.5,0);
  		\draw (f2.east) -- +(.5,0);
  		\node[link] at ($(f1.west)!.5!(f2.west)+(-1,0)$) (dot) {};
  		\draw (f1.west) to[out=180, in=60] (dot);
  		\draw (f2.west) to[out=180, in=-60] (dot);
  		\draw (dot) to +(-1,0);
    \end{tikzpicture}
    &
	    \begin{tikzpicture}[inner WD, baseline=-.5ex]
	      \node[oshelll] (f) {$f$};
	      \node[oshellr, right=.75 of f] (g) {$f$};
	      \draw (f.west) -- +(-.5, 0);
	      \draw (f) -- (g);
	      \draw (g.east) -- +(.5,0);
      \end{tikzpicture}	
      $\le$
	    \begin{tikzpicture}[inner WD, baseline=-.5ex]
	      \draw (-15pt,0) to (15pt,0);
      \end{tikzpicture}	
    \\
      \toprule[.1pt]
      counit homomorphism \textsc{(h2)} & total \textsc{(a2)} 
      \\
      \begin{tikzpicture}[inner WD, baseline=-.5ex]
	      \node[oshellr] (f) {$f$};
	      \node[link, right=.5 of f] (dot) {};
	      \draw (f.west) -- +(-.5, 0);
	      \draw (f) to (dot);
      \end{tikzpicture}	
      $=$
	\begin{tikzpicture}[inner WD, baseline=-.5ex]
    \node[link] (dot) {};
	  \draw (dot) to +(-10pt, 0);
	\end{tikzpicture}
      &
	    \begin{tikzpicture}[inner WD, baseline=-.5ex]
	      \draw (-15pt,0) to (15pt,0);
      \end{tikzpicture}	
      $\le$
	    \begin{tikzpicture}[inner WD, baseline=-.5ex]
	      \node[oshellr] (f) {$f$};
	      \node[oshelll, right=.75 of f] (g) {$f$};
	      \draw (f.west) -- +(-.5, 0);
	      \draw (f) -- (g);
	      \draw (g.east) -- +(.5,0);
      \end{tikzpicture}	
      \\
		\toprule[.1pt]
      multiplication homomorphism \textsc{(h3)} & co-deterministic \textsc{(a3)} 
      \\
    \begin{tikzpicture}[inner WD, baseline=-2.2ex]
    	\node[oshellr] (f1) {$f$};
  		\node[oshellr, below=.5 of f1] (f2) {$f$};
  		\draw (f1.west) -- +(-.5, 0);
  		\draw (f2.west) -- +(-.5, 0);
  		\node[link] at ($(f1.east)!.5!(f2.east)+(1,0)$) (dot) {};
  		\draw (f1.east) to[out=0, in=120] (dot);
  		\draw (f2.east) to[out=0, in=-120] (dot);
  		\draw (dot) to +(1,0);
    \end{tikzpicture}
    $=$
    \begin{tikzpicture}[inner WD,baseline=-2.2ex]
    	\coordinate (f1);
  		\coordinate[below=2 of f1] (f2);
  		\node[link] at ($(f1.east)!.5!(f2.east)+(1,0)$) (dot) {};
    	\node[oshellr, right=.5 of dot] (f0) {$f$};
      \draw (f0.east) to +(1,0);
      \draw (f0) to (dot);
  		\draw (f1.east) to[out=0, in=120] (dot);
  		\draw (f2.east) to[out=0, in=-120] (dot);
  		\draw (f1.west) -- +(-.5,0);
  		\draw (f2.west) -- +(-.5,0);
    \end{tikzpicture}
    &
	    \begin{tikzpicture}[inner WD, baseline=-.5ex]
	      \node[oshellr] (f) {$f$};
	      \node[oshelll, right=.75 of f] (g) {$f$};
	      \draw (f.west) -- +(-.5, 0);
	      \draw (f) -- (g);
	      \draw (g.east) -- +(.5, 0);
      \end{tikzpicture}	
      $\le$
	    \begin{tikzpicture}[inner WD, baseline=-.5ex]
	      \draw (-15pt,0) to (15pt,0);
	    \end{tikzpicture}	
      \\
      \toprule[.1pt]
      unit homomorphism \textsc{(h4)} & co-total \textsc{(a4)} 
      \\
      \begin{tikzpicture}[inner WD, baseline=-.5ex]
	      \node[oshellr] (f) {$f$};
	      \node[link, left=.5 of f] (dot) {};
	      \draw (f.east) -- +(.5, 0);
	      \draw (f) to (dot);
      \end{tikzpicture}	
      $=$
	\begin{tikzpicture}[inner WD, baseline=-.5ex]
    \node[link] (dot) {};
	  \draw (dot) to +(10pt, 0);
	\end{tikzpicture}
      &
	    \begin{tikzpicture}[inner WD, baseline=-.5ex]
	      \draw (-15pt,0) to (15pt,0);
      \end{tikzpicture}	
      $\le$
	    \begin{tikzpicture}[inner WD, baseline=-.5ex]
	      \node[oshelll] (f) {$f$};
	      \node[oshellr, right=.75 of f] (g) {$f$};
	      \draw (f.west) -- +(-.5, 0);
	      \draw (f) -- (g);
	      \draw (g.east) -- +(.5,0);
      \end{tikzpicture}	
      \\
      \toprule[1pt]
    \end{tabular}
  \end{center}
  These four implications (\textbf{\textsc{h}}\;$\Rightarrow$\;\textbf{\textsc{a}}) have various converses, as follows:
  \begin{description}
  	\item[($\textsc{a2}\Rightarrow\textsc{h2}$):] if $f$ is total then it is a counit homomorphism.
		\item[($\textsc{a4}\Rightarrow\textsc{h4}$):] if $f$ is co-total then it is a unit homomorphism.
		\item[($\textsc{a1}\wedge\textsc{a2}\Rightarrow\textsc{h1}$):] if $f$ is \emph{both} total and deterministic then it is a comult.\ homomorphism.
		\item[($\textsc{a3}\wedge\textsc{a4}\Rightarrow\textsc{h3}$):] if $f$ is \emph{both} co-deterministic and co-total it is a mult.\ homomorphism.
		\end{description}
\end{proposition}
\begin{proof}
  Recall that by definition of prerelational po-category, $f$ is already a lax comonoid homomorphism \eqref{eqn.lax_hom_prerel}, i.e.\ \textsc{h1} and \textsc{h2} automatically hold; we will mark uses of this fact with \textsc{(lx)}. To see that comultiplication homomorphisms are deterministic (\textsc{h1}$\Rightarrow$\textsc{a1}) and counit homomorphisms are total (\textsc{h2}$\Rightarrow$\textsc{a2}), observe respectively (see \cref{rem.frob_means_freedom})
  \[
    \begin{aligned}
	\begin{tikzpicture}[unoriented WD, font=\small]
	  \node (P1) {
	    \begin{tikzpicture}[inner WD]
	      \node[oshelll] (f) {$f$};
	      \node[oshellr, right=.75 of f] (g) {$f$};
	      \draw (f.west) -- +(-.5, 0);
	      \draw (f) -- (g);
	      \draw (g.east) -- +(.5,0);
	    \end{tikzpicture}	
	  };
	  \node[right=1.5 of P1] (P2) {
		\begin{tikzpicture}[inner WD]
        \node[link] (mid) {}; 
	      \node[oshelld, above=.5 of mid] (f) {$f$};
	      \node[link, above=.5 of f] (dot) {};
	      \draw (mid) -- +(-10pt, 0);
	      \draw (mid) -- +(10pt, 0);
	      \draw (dot) to (f);
	      \draw (f) to (mid);
	    \end{tikzpicture}	
	  };
	  \node[right=1.5 of P2] (P3) {
		\begin{tikzpicture}[inner WD]
        \node[link] (mid) {}; 
	      \node[link, above=.5 of mid] (dot) {};
	      \draw (mid) -- +(-10pt, 0);
	      \draw (mid) -- +(10pt, 0);
	      \draw (mid) to (dot);
	    \end{tikzpicture}	
	  };
	  \node[right=1.5 of P3] (P4) {
	    \begin{tikzpicture}[inner WD]
	      \draw (-8pt,0) to (8pt,0);
	    \end{tikzpicture}	
	  };
	  \node (i) at ($(P1.east)!.5!(P2.west)$) {$=$};
	  \node[above=-.5 of i] {\textsc{\tiny (h1)}};
	  \node (ii) at ($(P2.east)!.5!(P3.west)$) {$\leq$};
	  \node[above=-.5 of ii] {\textsc{\tiny (lx)}};
	  \node (iii) at ($(P3.east)!.5!(P4.west)$) {$=$};
  \end{tikzpicture}
\end{aligned}
  \quad
  \mbox{and}
  \quad
    \begin{aligned}
	\begin{tikzpicture}[unoriented WD, font=\small]
	  \node (P1) {
	    \begin{tikzpicture}[inner WD]
	      \draw (-8pt,0) to (8pt,0);
	    \end{tikzpicture}	
	  };
	  \node[right=1.5 of P1] (P2) {
		\begin{tikzpicture}[inner WD]
        \node[link] (mid) {}; 
	      \node[link, above=.5 of mid] (dot) {};
	      \draw (mid) -- +(-10pt, 0);
	      \draw (mid) -- +(10pt, 0);
	      \draw (mid) to (dot);
	    \end{tikzpicture}	
	  };
	  \node[right=1.5 of P2] (P3) {
		\begin{tikzpicture}[inner WD]
        \node[link] (mid) {}; 
	      \node[oshellu, above=.5 of mid] (f) {$f$};
	      \node[link, above=.5 of f] (dot) {};
	      \draw (mid) -- +(-10pt, 0);
	      \draw (mid) -- +(10pt, 0);
	      \draw (dot) to (f);
	      \draw (f) to (mid);
	    \end{tikzpicture}	
	  };
	  \node[right=1.5 of P3] (P4) {
	    \begin{tikzpicture}[inner WD]
	      \node[oshellr] (f) {$f$};
	      \node[oshelll, right=.75 of f] (g) {$f$};
	      \draw (f.west) -- +(-.5, 0);
	      \draw (f) -- (g);
	      \draw (g.east) -- +(.5,0);
	    \end{tikzpicture}	
	  };
	  \node (i) at ($(P1.east)!.5!(P2.west)$) {$=$};
	  \node (ii) at ($(P2.east)!.5!(P3.west)$) {$=$};
	  \node (iii) at ($(P3.east)!.5!(P4.west)$) {$\leq$};
	  \node[above=-.5 of ii] {\textsc{\tiny (h2)}};
	  \node[above=-.5 of iii] {\textsc{\tiny (lx)}};
	\end{tikzpicture}
\end{aligned}
  \]
  The arguments for the remaining two rows, ($\textsc{h3}\Rightarrow\textsc{a3}$) and ($\textsc{h4}\Rightarrow\textsc{a4}$), are dual.

  For the converses ($\textsc{a2}\Rightarrow\textsc{h2}$) and ($\textsc{a1}\wedge\textsc{a2}\Rightarrow\textsc{h1}$), we have 
  \[
	\begin{tikzpicture}[unoriented WD, font=\small,baseline=(P1)]
	  \node (P1) {
	\begin{tikzpicture}[inner WD]
    \node[link] (dot) {};
	  \draw (dot) to +(-10pt, 0);
	\end{tikzpicture}
	  };
	  \node[right=2 of P1] (P2) {
	    \begin{tikzpicture}[inner WD]
	      \node[oshellr] (f) {$f$};
	      \node[oshelll, right=.75 of f] (g) {$f$};
	      \node[link, right=.5 of g] (dot) {};
	      \draw (f.west) -- +(-.5, 0);
	      \draw (f) -- (g);
	      \draw (g) -- (dot);
	    \end{tikzpicture}	
	  };
	  \node[right=2 of P2] (P3) {
	    \begin{tikzpicture}[inner WD]
	      \node[oshellr] (f) {$f$};
	      \node[link, right=.5 of f] (dot) {};
	      \draw (f.west) -- +(-.5, 0);
	      \draw (f) to (dot);
	    \end{tikzpicture}	
	  };
	  \node (i) at ($(P1.east)!.5!(P2.west)$) {$\leq$};
	  \node (ii) at ($(P2.east)!.5!(P3.west)$) {$\leq$};
	  \node[above=-.5 of i] {\textsc{\tiny (a2)}};
	  \node[above=-.5 of ii] {\textsc{\tiny (lx)}};
  \end{tikzpicture}
  \]
  and
  \[
  \begin{tikzpicture}[font=\small]
  	\node (P1) {
    \begin{tikzpicture}[inner WD]
    	\node[oshellr] (f1) {$f$};
  		\node[oshellr, below=.5 of f1] (f2) {$f$};
  		\draw (f1.east) -- +(1,0);
  		\draw (f2.east) -- +(1,0);
  		\node[link] at ($(f1.west)!.5!(f2.west)+(-1,0)$) (dot) {};
  		\draw (f1.west) to[out=180, in=60] (dot);
  		\draw (f2.west) to[out=180, in=-60] (dot);
  		\draw (dot) to +(-1,0);
    \end{tikzpicture}
    };
  	\node[right=1 of P1] (P2) {
    \begin{tikzpicture}[inner WD]
    	\node[oshellr] (f1) {$f$};
  		\node[oshellr, below=.5 of f1] (f2) {$f$};
  		\node[link] at ($(f1.west)!.5!(f2.west)+(-1,0)$) (dot) {};
    	\node[oshelll, left=.5 of dot] (g) {$f$};
    	\node[oshellr, left=.5 of g] (f0) {$f$};
      \draw (f0.west) to +(-1,0);
      \draw (f0) to (g);
      \draw (g) to (dot);
  		\draw (f1.east) -- +(1,0);
  		\draw (f2.east) -- +(1,0);
  		\draw (f1.west) to[out=180, in=60] (dot);
  		\draw (f2.west) to[out=180, in=-60] (dot);
    \end{tikzpicture}
    };
  	\node[right=1 of P2] (P3) {
    \begin{tikzpicture}[inner WD]
    	\node[oshellr] (f1) {$f$};
  		\node[oshellr, below=.5 of f1] (f2) {$f$};
    	\node[oshelll, left=.5 of f1] (g1) {$f$};
    	\node[oshelll, left=.5 of f2] (g2) {$f$};
  		\node[link] at ($(g1.west)!.5!(g2.west)+(-1,0)$) (dot) {};
    	\node[oshellr, left=.5 of dot] (f0) {$f$};
      \draw (f0.west) to +(-1,0);
      \draw (f0) to (dot);
  		\draw (g1.west) to[out=180, in=60] (dot);
  		\draw (g2.west) to[out=180, in=-60] (dot);
      \draw (g1) to (f1);
      \draw (g2) to (f2);
  		\draw (f1.east) -- +(1,0);
  		\draw (f2.east) -- +(1,0);
    \end{tikzpicture}
    };
  	\node[right=1 of P3] (P4) {
    \begin{tikzpicture}[inner WD]
    	\coordinate (f1);
  		\coordinate[below=2 of f1] (f2);
  		\node[link] at ($(f1.west)!.5!(f2.west)+(-1,0)$) (dot) {};
    	\node[oshellr, left=.5 of dot] (f0) {$f$};
      \draw (f0.west) to +(-1,0);
      \draw (f0) to (dot);
  		\draw (f1.west) to[out=180, in=60] (dot);
  		\draw (f2.west) to[out=180, in=-60] (dot);
  		\draw (f1.east) -- +(1,0);
  		\draw (f2.east) -- +(1,0);
    \end{tikzpicture}
    };
	  \node (i) at ($(P1.east)!.5!(P2.west)$) {$\leq$};
	  \node (ii) at ($(P2.east)!.5!(P3.west)$) {$\leq$};
	  \node (iii) at ($(P3.east)!.5!(P4.west)$) {$\leq$};
	  \node[above=-.1 of i] {\textsc{\tiny (a2)}};
	  \node[above=-.1 of ii] {\textsc{\tiny (lx)}};
	  \node[above=-.1 of iii] {\textsc{\tiny (a1)}};
\end{tikzpicture}
\]
Again the arguments proving ($\textsc{a4}\Rightarrow\textsc{h4}$) and ($\textsc{a3}\wedge\textsc{a4}\Rightarrow\textsc{h3}$) are dual.
\end{proof}

\begin{corollary}[Characterization of left adjoints] \label{cor.leftadjs}
  Let $f\colon r \to s$ be a morphism in a prerelational po-category $\cc$. Then the following are equivalent.
  \begin{enumerate}[label=(\roman*)]
    \item $f$ is a left adjoint.
    \item $f$ is left adjoint to its transpose $f\tp \colon s \radjto r$.
    \item $f$ is deterministic and total.
    \item $f$ is a comonoid homomorphism.
  \end{enumerate}
  Similarly $f$ is a right adjoint iff it is right adjoint to its transpose, iff it is codeterminstic and cototal, iff it is a monoid homomorphism.
\end{corollary}
\begin{proof}
  The equivalence of (ii) and (iii) is by definition, and that of (ii) and (iv) is immediate from \cref{prop.adj_vs_hom}; also note that (ii) evidently implies (i). To see that (i) implies (iv), note that in the proof the homomorphism properties ($\textsc{a1}\wedge\textsc{a2}\Rightarrow\textsc{h1}\wedge\textsc{h2}$) in \cref{prop.adj_vs_hom}, only the existence of an `adjoint' was necessary; it did not matter that the morphism with the adjointness property was the transpose.
\end{proof}

\begin{notation}[Left adjoints]
  From now on, we shall write denote 
    $
    \begin{tikzpicture}[inner WD] 
      \node[oshellr] (f) {$f$};
      \draw (f.west) -- +(-.5,0) node[left] {$r$};
      \draw (f.east) -- +(.5,0)  node[right] {$s$};
    \end{tikzpicture}
  $
by
  $
    \begin{tikzpicture}[inner WD] 
      \node[funcr] (f) {$f$};
      \draw (f.west) -- +(-.5,0) node[left] {$r$};
      \draw (f.east) -- +(.5,0) node[right] {$s$};
    \end{tikzpicture}
  $
  when $f\colon r\ladjto s$ is known to be a left adjoint. By \cref{cor.leftadjs} we know its transpose $f\tp\colon s\radjto r$ is its right adjoint, and we denote it by
  $
    \begin{tikzpicture}[inner WD] 
      \node[funcl] (f) {$f$};
      \draw (f.west) -- +(-.5,0) node[left] {$s$};
      \draw (f.east) -- +(.5,0) node[right] {$r$};
    \end{tikzpicture}
  $.
\end{notation}

\begin{proposition}[Uniqueness of left adjoint comonoids]\label{prop.prerel_comon_unique}
  Let $c\in\cc$ be an object in a prerelational po-category. If $\delta\colon c \ladjto c \otimes c$ and $\epsilon\colon c \ladjto I$ are left adjoints such that $(c,\delta,\epsilon)$ is a comonoid, then $\epsilon = \epsilon_c$ and $\delta = \delta_c$.
\end{proposition}
\begin{proof}
  By the definition of supply, the counit $\epsilon_I$ supplied to the monoidal unit is simply the identity $\id_I$; see \cref{eqn.supply_commute_tensors}. Since left adjoints are comonoid homomorphisms by \cref{cor.leftadjs}, we have that $\epsilon = \epsilon \cp \id_I = \epsilon\cp\epsilon_I=\epsilon_c$ by \textsc{h2}. This implies the last equality in the chain:
  \[
  \begin{tikzpicture}[font=\small, baseline=(P1)]
  	\node (P1) {
    \begin{tikzpicture}[inner WD, shorten <=-1pt]
  		\node[funcr] (d) {$\delta$};
      \draw (d.180) -- +(-1,0);
  		\draw (d.25) -- +(1,0);
  		\draw (d.-25) -- +(1,0);
    \end{tikzpicture}
    };
  	\node[right=1 of P1] (P2) {
    \begin{tikzpicture}[inner WD, shorten <=-1pt]
  		\node[funcr] (d) {$\delta$};
      \draw (d.180) -- +(-1,0);
  		\draw (d.25) -- +(1,0) node[link] (dot1) {};
  		\draw (d.-25) -- +(1,0) node[link] (dot2) {};
      \draw (dot1) to[out=60, in=180] +(1,1) -- +(1,0);
      \draw (dot1) to[out=-60, in=180] +(1,-1) -- +(.1,0) node[link] {};
      \draw (dot2) to[out=60, in=180] +(1,1) -- +(.1,0) node[link] {};
      \draw (dot2) to[out=-60, in=180] +(1,-1) -- +(1,0);
    \end{tikzpicture}
    };
  	\node[right=1 of P2] (P3) {
    \begin{tikzpicture}[inner WD, shorten <=-1pt]
    	\node[funcr] (f1) {$\delta$};
  		\node[funcr, below=.5 of f1] (f2) {$\delta$};
  		\draw (f1.25) -- +(2,0);
  		\draw (f1.-25) -- +(1,0) node[link] {};
  		\draw (f2.25) -- +(1,0) node[link] {};
  		\draw (f2.-25) -- +(2,0);
  		\node[link] at ($(f1.west)!.5!(f2.west)+(-1,0)$) (dot) {};
  		\draw (f1.west) to[out=180, in=60] (dot);
  		\draw (f2.west) to[out=180, in=-60] (dot);
  		\draw (dot) to +(-1,0);
    \end{tikzpicture}
    };
  	\node[right=1 of P3] (P4) {
    \begin{tikzpicture}[inner WD]
    	\coordinate (f1);
  		\coordinate[below=2 of f1] (f2);
  		\node[link] at ($(f1.west)!.5!(f2.west)+(-1,0)$) (dot) {};
      \draw (dot) -- +(-1,0);
  		\draw (f1.west) to[out=180, in=60] (dot);
  		\draw (f2.west) to[out=180, in=-60] (dot);
  		\draw (f1.east) -- +(1,0);
  		\draw (f2.east) -- +(1,0);
    \end{tikzpicture}
    };
	  \node (i) at ($(P1.east)!.5!(P2.west)$) {$=$};
	  \node (ii) at ($(P2.east)!.5!(P3.west)$) {$=$};
	  \node (iii) at ($(P3.east)!.5!(P4.west)$) {$=$};
	  \node[above=-.1 of ii] {\textsc{\tiny (H1)}};
\end{tikzpicture}
\qedhere
\]
\end{proof}

\begin{proposition}[Left adjoints are discretely ordered] \label{prop.ladj_discrete}
  Suppose we have left adjoints $f$, $g$ in a prerelational po-category such that $f \le g$. Then $f = g$.
\end{proposition}
\begin{proof}
  Note that $f \le g$ immediately implies $f\tp \le g\tp$, because composition---e.g.\ composing with caps and cups---is monotonic. Then by adjointness we have that
  \[
  \begin{tikzpicture}[unoriented WD, font=\small]
    \node (P1) {
      \begin{tikzpicture}[inner WD]
        \node[funcr] (f) {$g$};
        \draw (f) -- +(-10pt, 0);
        \draw (f) -- +(10pt,0);
      \end{tikzpicture}
    };
    \node[right=2 of P1] (P2) {
      \begin{tikzpicture}[inner WD]
        \node[funcr] (f) {$g$};
        \node[funcl, right=.75 of f] (g) {$f$};
        \node[funcr, right=.75 of g] (h) {$f$};
        \draw (f) -- +(-10pt, 0);
        \draw (f) -- (g);
        \draw (g) -- (h);
        \draw (h) -- +(10pt,0);
      \end{tikzpicture}
    };
    \node[right=2 of P2] (P3) {
      \begin{tikzpicture}[inner WD]
        \node[funcr] (f) {$g$};
        \node[funcl, right=.75 of f] (g) {$g$};
        \node[funcr, right=.75 of g] (h) {$f$};
        \draw (f) -- +(-10pt, 0);
        \draw (f) -- (g);
        \draw (g) -- (h);
        \draw (h) -- +(10pt,0);
      \end{tikzpicture}
    };
    \node[right=2 of P3] (P4) {
      \begin{tikzpicture}[inner WD]
        \node[funcr] (f) {$f$};
        \draw (f) -- +(-10pt, 0);
        \draw (f) -- +(10pt,0);
      \end{tikzpicture}
    };
    \node at ($(P1.east)!.5!(P2.west)$) {$\le$};
    \node at ($(P2.east)!.5!(P3.west)$) {$\le$};
    \node at ($(P3.east)!.5!(P4.west)$) {$\le$};
  \end{tikzpicture}.
  \qedhere
  \]
\end{proof}

\cref{lem.lrl_is_l}, which we will use often, has a nearly identical proof.

\begin{lemma}\label{lem.lrl_is_l}
Suppose $f$ is a left adjoint in a prerelational po-category. Then
\begin{equation}\label{eqn.lrl_is_l}
  \begin{tikzpicture}[unoriented WD, font=\small]
    \node (P1) {
      \begin{tikzpicture}[inner WD]
        \node[funcr] (f) {$f$};
        \node[funcl, right=.65 of f] (g) {$f$};
        \node[funcr, right=.75 of g] (h) {$f$};
        \draw (f) -- +(-10pt, 0);
        \draw (f) -- (g);
        \draw (g) -- (h);
        \draw (h) -- +(10pt,0);
      \end{tikzpicture}
    };
    \node[right=2 of P1] (P2) {
      \begin{tikzpicture}[inner WD]
        \node[funcr] (f) {$f$};
        \draw (f) -- +(-10pt, 0);
        \draw (f) -- +(10pt,0);
      \end{tikzpicture}
    };
    \node at ($(P1.east)!.5!(P2.west)$) {$=$};
  \end{tikzpicture}
\end{equation}
\end{lemma}

The following lemma, \cref{lemma.square_BC_transform}, states that a square of left adjoints
\[
  \begin{tikzcd}
    a \ar[r,"y"] \ar[d,"x"'] &  b \ar[d,"g"] \\
    c \ar[r,"f"'] & d
  \end{tikzcd}
\]
commutes if and only if the composite relation $c \to b$ via $a$ is less than that via $d$.

\begin{lemma}\label{lemma.square_BC_transform}
  In a prerelational po-category, we have
  \[
	\begin{tikzpicture}[unoriented WD, font=\small,baseline=(P1)]
	  \node (P1) {
	    \begin{tikzpicture}[inner WD]
	      \node[funcr] (f) {$x$};
	      \node[funcr, right=.75 of f] (g) {$f$};
	      \draw (f) -- +(-15pt, 0);
	      \draw (f) -- (g);
	      \draw (g) -- +(15pt,0);
	    \end{tikzpicture}	
	  };
	  \node[right=2 of P1] (P2) {
	    \begin{tikzpicture}[inner WD]
	      \node[funcr] (f) {$y$};
	      \node[funcr, right=.75 of f] (g) {$g$};
	      \draw (f) -- +(-15pt, 0);
	      \draw (f) -- (g);
	      \draw (g) -- +(15pt,0);
	    \end{tikzpicture}	
	  };
	  \node (i) at ($(P1.east)!.5!(P2.west)$) {$=$};
  \end{tikzpicture}    
  \qquad
  \mbox{iff}
  \qquad
	\begin{tikzpicture}[unoriented WD, font=\small,baseline=(P1)]
	  \node (P1) {
	    \begin{tikzpicture}[inner WD]
	      \node[funcl] (f) {$x$};
	      \node[funcr, right=.75 of f] (g) {$y$};
	      \draw (f) -- +(-15pt, 0);
	      \draw (f) -- (g);
	      \draw (g) -- +(15pt,0);
	    \end{tikzpicture}	
	  };
	  \node[right=2 of P1] (P2) {
	    \begin{tikzpicture}[inner WD]
	      \node[funcr] (f) {$f$};
	      \node[funcl, right=.75 of f] (g) {$g$};
	      \draw (f) -- +(-15pt, 0);
	      \draw (f) -- (g);
	      \draw (g) -- +(15pt,0);
	    \end{tikzpicture}	
	  };
	  \node (i) at ($(P1.east)!.5!(P2.west)$) {$\leq$};
  \end{tikzpicture}    
  \]
\end{lemma}
\begin{proof}
  Suppose $x\cp f= y \cp g$. Then $x\tp \cp y \le f \cp f\tp \cp x\tp \cp y = f \cp g\tp \cp y\tp \cp y \le f \cp g\tp$.
  Conversely, suppose $x\tp \cp y \le f \cp g\tp$. Then $x \cp f \le y \cp y\tp \cp x \cp f \le y \cp g \cp f\tp \cp f \le y \cp g$.
  Since both $x\cp f$ and $y \cp g$ are left adjoints, they are equal by \cref{prop.ladj_discrete}.
\end{proof}

\section{Epis and monos in a prerelational po-category} \label{sec.epis_monos}

Morphisms in a regular category have image factorizations. In this section we explore epis and monos in prerelational categories.

\begin{definition}[Epis and monos in a po-category]\label{def.epis_monos_2_cat}
We say that a morphism $m\colon r\to s$ in a po-category $\cc$ is a  \emph{monomorphism} iff it is a monomorphism in the underlying 1-category, i.e.\ if the function $(- \cp m)\colon\cc(q,r)\to\cc(q,s)$ is injective (the \emph{cancellation property} holds). Similary, we say that $m$ is a \emph{split monomorphism} in $\cc$ if there exists $n$ such that $m \cp n = \id$. We define (split) epimorphisms in $\cc$ dually.
\end{definition}

\begin{proposition}\label{prop.characterize_monos}
  Let $m\colon r \ladjto s$ be a left adjoint in a prerelational po-category $\cc$. Then the following are equivalent.
  \begin{enumerate}[label=(\roman*)]
    \item $m$ is a monomorphism in $\cc$.
    \item $m$ is a split monomorphism in $\cc$.
    \item the unit of the $m, m\tp$ adjunction is the identity: 
    $
  \begin{tikzpicture}[unoriented WD, font=\small, baseline=(P2)]
    \node (P1) {
      \begin{tikzpicture}[inner WD]
        \draw (-10pt,0) -- (10pt,0);
      \end{tikzpicture}
    };
    \node[right=2 of P1] (P2) {
      \begin{tikzpicture}[inner WD]
        \node[funcr] (f) {$m$};
        \node[funcl, right=.75 of f] (g) {$m$};
        \draw (f) -- +(-10pt, 0);
        \draw (f) -- (g);
        \draw (g) -- +(10pt,0);
      \end{tikzpicture}
    };
    \node at ($(P1.east)!.5!(P2.west)$) {$=$};
  \end{tikzpicture}
  $.
    \item $m$ is co-deterministic.
  \end{enumerate}
  Dually, a left adjoint $e$ is epi in $\cc$ iff it is split epi in $\cc$ iff the unit of the $e,e\tp$ adjunction is the identity iff it is co-total.
\end{proposition}
\begin{proof}
  (iii) implies (ii) implies (i) trivially, while (i) implies (iii) by applying the cancellation property to \cref{eqn.lrl_is_l}. (iii) implies (iv) trivially, and (iv) implies (iii) because $m$ is a left adjoint, hence total; see \cref{eqn.show_adjoints}.
\end{proof}


\begin{corollary} \label{prop.extremal_epis}
  Let $e$ be a left adjoint in a prerelational po-category $\cc$. If $e$ is epi in $\cc$, then it is extremal epi in $\ladj(\cc)$.
\end{corollary}
\begin{proof}
  Suppose that $e$ is epi and $f$ and $m$ are left adjoints such that $m$ is a monomorphism and $e = f \cp m$. By definition of extremal epi, we need to show that $m$ is an iso. It already satisfies \textsc{a1}, \textsc{a2}, and \textsc{a3}; we prove it satisfies \textsc{a4} using \cref{prop.characterize_monos}:
  \[
  \raisebox{-.2cm}{
  \begin{tikzpicture}
  	\node (p1) {
		\begin{tikzpicture}[inner WD]
			\node[funcr] (m) {$m$};
			\node[link, left=.5 of m] (link) {};
			\draw (m.west) -- (link);
			\draw (m) -- +(10pt,0);
		\end{tikzpicture}
		};
		\node (p2) [right=.5 of p1] {
		\begin{tikzpicture}[inner WD]
			\node[funcr] (m) {$m$};
			\node[link, left=.5 of m] (link) {};
			\draw (m.west) -- (link);
			\node[funcl, right=.75 of m] (L) {$e$};
			\draw (m.east) -- (L.west);
			\node[funcr, right=.75 of L] (e) {$e$};
			\draw (e) -- +(10pt,0);
			\draw (L.east) -- (e.west);
		\end{tikzpicture}		
		};
		\node (p3) [right=.5 of p2] {
		\begin{tikzpicture}[inner WD]
			\node[funcr] (m) {$m$};
			\node[link, left=.5 of m] (link) {};
			\draw (m.west) -- (link);
			\node[funcl, right=.75 of m] (L) {$m$};
			\draw (m.east) -- (L.west);
			\node[funcl, right=.75 of L] (f) {$f$};
			\node[funcr, right=.75 of f] (e) {$e$};
			\draw (L.east) -- (f.west);
			\draw (f.east) -- (e.west);
			\draw (e) -- +(10pt,0);
		\end{tikzpicture}		
		};
		\node (p4) [right=.5 of p3] {
		\begin{tikzpicture}[inner WD]
			\node[funcl] (m) {$f$};
			\node[link, left=.5 of m] (link) {};
			\draw (m.west) -- (link);
			\node[funcr, right=.75 of m] (L) {$e$};
			\draw (m.east) -- (L.west);
			\draw (L) -- +(10pt,0);
		\end{tikzpicture}		
		};
		\node (p5) [right=.5 of p4] {
		\begin{tikzpicture}[inner WD]
			\node[funcr] (m) {$e$};
			\node[link, left=.5 of m] (link) {};
			\draw (m.west) -- (link);
			\draw (m) -- +(10pt,0);
		\end{tikzpicture}		
		};
		\node (p6) [right=.5 of p5] {
		\begin{tikzpicture}[inner WD]
			\node[link] (link) {};
			\draw (link) -- +(10pt,0);
		\end{tikzpicture}				
		};
		\node at ($(p1.east)!.5!(p2.west)$) {$=$};
		\node at ($(p2.east)!.5!(p3.west)$) {$=$};
		\node at ($(p3.east)!.5!(p4.west)$) {$=$};
		\node at ($(p4.east)!.5!(p5.west)$) {$=$};
		\node at ($(p5.east)!.5!(p6.west)$) {$=$};
  \end{tikzpicture}
  }
  \qedhere
  \]
\end{proof}

\begin{proposition} \label{prop.condition3}
  Let $f\colon r \to s$ in $\cc$ have tabulation $f= f_R \cp f_L$, and let $\hat{f}$ denote the associated span. Then Condition (iii) in \cref{def.tabulation} holds iff $\hat{f}$ is monic in $\cc$.
\end{proposition}
\begin{proof}
  Note $\delta_{r}$, $f_L$, and $f_R\tp$ (by \cref{cor.leftadjs}) are left adjoints. Since left adjoints are closed under monoidal product and composition (\cref{prop.ladj_smc}), $\hat{f}=\delta_{r}\cp(f_L\otimes f_R\tp)$ is also left adjoint, with right adjoint equal to its transpose $(f_L\tp\otimes f_R)\cp\mu_{r}$. Thus by \cref{prop.characterize_monos}, Condition (iii) is equivalent to the fact that $\hat{f}$ is monic.
\end{proof}

\section{$\ladj(\rr)$ is regular}\label{sec.ladj_on_objects}

Fix a relational po-category $(\rr,I,\otimes)$ (see \cref{def.reg2cat}). Our goal in this section is to prove that its category $\ladj(\rr)$ of left adjoints is regular.


Thinking of $\rr$ as a po-category of relations, morphisms $x\colon I\to r$ represent subobjects $|x|\ss r$. When the image of a morphism $f$ is contained in $|x|$, \cref{prop.restrict_codomain} tells us how to factor $f$ through the inclusion $|x|\ss r$.

\begin{proposition} \label{prop.restrict_codomain}
  For any map $x \colon I \to r$, let $x=x_R\cp x_L$ be a tabulation, and let $f\colon s \ladjto r$ be a left adjoint. If
  $
  \begin{tikzpicture}[baseline=(p1.-10)]
  	\node (p1) {
		\begin{tikzpicture}[inner WD]
			\node[funcr] (f) {$f$};
			\node[link, left=.5 of m] (link) {};
			\draw (f.west) -- (link);
			\draw (f) -- +(10pt,0);
		\end{tikzpicture}		
		};
		\node (p2) [right=.3 of p1] {
		\begin{tikzpicture}[inner WD]
			\node[oshellr] (x) {$x$};
			\draw (x.east) -- +(.5,0);
		\end{tikzpicture}
		};
	  \node at ($(p1.east)!.5!(p2.west)$) {$\leq$};
  \end{tikzpicture}
  $, i.e.\ $\eta_s\cp f\leq x$, then:
  \begin{enumerate}[label=(\roman*)] 
    \item $x_L$ is a monomorphism in $\rr$.
    \item 
    $
	  \begin{tikzpicture}[inner WD, baseline=(f.-20)]
			\node[funcr] (f) {$f$};
			\node[funcl, right=.75 of f] (x) {$x_L$};
			\draw (f.west) -- +(-.5,0);
			\draw (f.east) -- (x.west);
			\draw (x.east) -- +(.5,0);
    \end{tikzpicture}
    $
    is a left adjoint.
    \item
    $
	  \begin{tikzpicture}[baseline=(p1.0)]
	  \node (p1) {
  	  \begin{tikzpicture}[inner WD]
  			\node[funcr] (f) {$f$};
  			\node[funcl, right=.75 of f] (x) {$x_L$};
  			\node[funcr, right=.75 of x] (x') {$x_L$};
  			\draw (f.west) -- +(-.5,0);
  			\draw (f.east) -- (x.west);
  			\draw (x.east) -- (x'.west);
  			\draw (x'.east) -- +(.5,0);
      \end{tikzpicture}
    };
	  \node (p2) [right=.3 of p1] {
  	  \begin{tikzpicture}[inner WD]
  			\node[funcr] (f) {$f$};
				\draw (f.west) -- +(-.5, 0);	  
				\draw (f.east) -- +(.5, 0);	  
	  	\end{tikzpicture}
		};
	  \node at ($(p1.east)!.5!(p2.west)$) {$=$};
    \end{tikzpicture}
    $.
  \end{enumerate}
  Moreover, in the case of equality
    $
  \begin{tikzpicture}[baseline=(p1.-10)]
  	\node (p1) {
		\begin{tikzpicture}[inner WD]
			\node[funcr] (f) {$f$};
			\node[link, left=.5 of m] (link) {};
			\draw (f.west) -- (link);
			\draw (f) -- +(10pt,0);
		\end{tikzpicture}		
		};
		\node (p2) [right=.3 of p1] {
		\begin{tikzpicture}[inner WD]
			\node[oshellr] (x) {$x$};
			\draw (x.east) -- +(.5,0);
		\end{tikzpicture}
		};
	  \node at ($(p1.east)!.5!(p2.west)$) {$=$};
  \end{tikzpicture}
  $, i.e.\ $\eta_s\cp f=x$, we have
  \begin{enumerate}
    \item[(iv)]     
    $
	  \begin{tikzpicture}[inner WD, baseline=(f.-20)]
			\node[funcr] (f) {$f$};
			\node[funcl, right=.75 of f] (x) {$x_L$};
			\draw (f.west) -- +(-.5,0);
			\draw (f.east) -- (x.west);
			\draw (x.east) -- +(.5,0);
    \end{tikzpicture}
    $
 is an extremal epimorphism in $\rr$.
  \end{enumerate}
\end{proposition}
\begin{proof}
 Given any morphism $x\colon I \to r$, its tabulation gives an object $|x|$, a right adjoint $x_R\colon I \radjto |x|$, and a left adjoint $x_L\colon |x| \ladjto r$. Note that since every right adjoint is a unit homomorphism, the unit $\eta_{|x|}$ is in fact the unique right adjoint $I \radjto |x|$. Hence we write
  \[
  \begin{tikzpicture}[baseline=(p1.-10)]
  	\node (p1) {
		\begin{tikzpicture}[inner WD]
			\node[oshellr] (x) {$x$};
			\draw (x.east) -- +(.5,0);
		\end{tikzpicture}
		};
		\node (p2) [right=.3 of p1] {
		\begin{tikzpicture}[inner WD]
			\node[funcr] (f) {$x_L$};
			\node[link, left=.5 of m] (link) {};
			\draw (f.west) -- (link);
			\draw (f) -- +(10pt,0);
		\end{tikzpicture}		
		};
	  \node at ($(p1.east)!.5!(p2.west)$) {$=$};
  \end{tikzpicture}
  \] 
  and by unitality and the definition \eqref{eqn.tabulation} of tabulation, $x_L$ is monic. This proves (i).

  By \cref{cor.leftadjs,prop.characterize_monos}, since $x_L$ is monic, ${x_L}\tp$ is deterministic---so the composite $f \cp {x_L}\tp$ is also deterministic---and it remains to see that $f \cp {x_L}\tp$ is total:
  \[
  \begin{tikzpicture}
  	\node (p1) {
		\begin{tikzpicture}[inner WD]
			\node[funcr] (f1) {$f$};
			\node[funcl, right=.5 of f1] (f2) {$x_L$};
			\node[link, right=.5 of f2] (link) {};
			\draw (f1) -- +(-10pt,0);
			\draw (f1) -- (f2);
			\draw (f2) -- (link);
		\end{tikzpicture}
		};
		\node (p2) [right=.5 of p1] {
		\begin{tikzpicture}[inner WD]
			\node[funcr] (f1) {$f$};
			\node[oshelll, right=.5 of f1] (f2) {$x$};
			\draw (f1) -- +(-10pt,0);
			\draw (f1) -- (f2);
		\end{tikzpicture}		
		};
		\node (p3) [right=.5 of p2] {
		\begin{tikzpicture}[inner WD]
			\node[funcr] (f1) {$f$};
			\node[funcl, right=.5 of f1] (f2) {$f$};
			\node[link, right=.5 of f2] (link) {};
			\draw (f1) -- +(-10pt,0);
			\draw (f1) -- (f2);
			\draw (f2) -- (link);
		\end{tikzpicture}
		};
		\node (p4) [right=.5 of p3] {
		\begin{tikzpicture}[inner WD]
			\node[link] (link) {};
			\draw (link) -- +(-10pt,0);
		\end{tikzpicture}		
		};
		\node (p5) [right=.5 of p4] {
		};
    \node at ($(p1.east)!.5!(p2.west)$) {$=$};
		\node at ($(p2.east)!.5!(p3.west)$) {$\ge$};
		\node at ($(p3.east)!.5!(p4.west)$) {$\ge$};
  \end{tikzpicture}
  \]
  Thus $f \cp {x_L}\tp$ is a left adjoint, proving (ii).

  For (iii), we have the following inequalities, using \cref{lem.lrl_is_l} in the final step:
  \[
  \begin{tikzpicture}
  	\node (p1) {
		\begin{tikzpicture}[inner WD]
			\node[funcr] (f1) {$f$};
			\node[funcl, right=.5 of f1] (f2) {$x_L$};
			\node[funcr, right=.5 of f2] (f3) {$x_L$};
			\draw (f1) -- +(-10pt,0);
			\draw (f1) -- (f2);
			\draw (f2) -- (f3);
			\draw (f3) -- +(10pt,0);
		\end{tikzpicture}
		};
		\node (p2) [right=.5 of p1] {
		\begin{tikzpicture}[inner WD]
			\node[funcr] (f1) {$f$};
			\node[link, right=.5 of f1] (link) {};
			\node[funcd, above=.5 of link] (f2) {$x_L$};
      \node[link, above=.5 of f2] (link2) {};
			\draw (f1) -- +(-10pt,0);
			\draw (f1) -- (link);
			\draw (f2) -- (link);
			\draw (f2) -- (link2);
      \draw (link) -- +(10pt,0);
    \end{tikzpicture}
		};
		\node (p3) [right=.5 of p2] {
		\begin{tikzpicture}[inner WD]
			\node[funcr] (f1) {$f$};
			\node[link, right=.5 of f1] (link) {};
			\node[oshelld, above=.5 of link] (f2) {$x$};
			\draw (f1) -- +(-10pt,0);
			\draw (f1) -- (link);
			\draw (f2) -- (link);
      \draw (link) -- +(10pt,0);
    \end{tikzpicture}
		};
		\node (p4) [right=.5 of p3] {
		\begin{tikzpicture}[inner WD]
			\node[funcr] (f1) {$f$};
			\node[link, right=.5 of f1] (link) {};
			\node[funcd, above=.5 of link] (f2) {$f$};
      \node[link, above=.5 of f2] (link2) {};
			\draw (f1) -- +(-10pt,0);
			\draw (f1) -- (link);
			\draw (f2) -- (link);
			\draw (f2) -- (link2);
      \draw (link) -- +(10pt,0);
    \end{tikzpicture}
		};
		\node (p5) [right=.5 of p4] {
		\begin{tikzpicture}[inner WD]
			\node[funcr] (f1) {$f$};
			\node[funcl, right=.5 of f1] (f2) {$f$};
			\node[funcr, right=.5 of f2] (f3) {$f$};
			\draw (f1) -- +(-10pt,0);
			\draw (f1) -- (f2);
			\draw (f2) -- (f3);
			\draw (f3) -- +(10pt,0);
		\end{tikzpicture}
		};
		\node (p6) [right=.5 of p5] {
		\begin{tikzpicture}[inner WD]
			\node[funcr] (f1) {$f$};
			\draw (f1) -- +(-10pt,0);
			\draw (f1) -- +(10pt,0);
		\end{tikzpicture}				
		};
		\node at ($(p1.east)!.5!(p2.west)$) {$=$};
		\node at ($(p2.east)!.5!(p3.west)$) {$=$};
		\node at ($(p3.east)!.5!(p4.west)$) {$\ge$};
		\node at ($(p4.east)!.5!(p5.west)$) {$=$};
		\node at ($(p5.east)!.5!(p6.west)$) {$=$};
  \end{tikzpicture}
  \]
Since the order on left adjoints is discrete (\cref{prop.ladj_discrete}), this implies (iii).

  Finally, for (iv) assume $x=\eta\cp f$. By \cref{prop.adj_vs_hom,prop.extremal_epis,prop.characterize_monos} it is enough to prove that $f \cp {x_L}\tp$ is a unit homomorphism. This follows from (i):
  \[
  \begin{tikzpicture}[baseline=(p1)]
		\node (p1) {
		\begin{tikzpicture}[inner WD]
			\node[funcr] (f1) {$f$};
			\node[funcl, right=.75 of f1] (f2) {$x_L$};
			\node[link, left=.5 of f1] (link) {};
      \draw (f2) -- +(10pt,0);
      \draw (f1) -- (f2);
			\draw (f1.west) -- (link);
		\end{tikzpicture}		
		};
		\node (p2) [right=.5 of p1] {
		\begin{tikzpicture}[inner WD]
			\node[funcr] (f1) {$x_L$};
			\node[funcl, right=.75 of f1] (f2) {$x_L$};
			\node[link, left=.5 of f1] (link) {};
			\draw (f2) -- +(10pt,0);
      \draw (f1) -- (f2);
			\draw (f1.west) -- (link);
		\end{tikzpicture}		
		};
		\node (p3) [right=.5 of p2] {
		\begin{tikzpicture}[inner WD]
			\node[link] (link) {};
			\draw (link) -- +(10pt,0);
		\end{tikzpicture}		
		};
    \node at ($(p1.east)!.5!(p2.west)$) {$=$};
		\node at ($(p2.east)!.5!(p3.west)$) {$=$};
  \end{tikzpicture}
  \qedhere
  \]
\end{proof}

\begin{corollary}[Factorization] \label{prop.factorization}
  Every morphism in $\ladj(\rr)$ factors as an extremal epi followed by a mono.
\end{corollary}
\begin{proof}
  Given $f$ in $\ladj(\rr)$, apply \cref{prop.restrict_codomain} in the case that $x = \eta \cp f$.
\end{proof}

\begin{corollary}\label{cor.extremal_epis}
  Let $e$ be a left adjoint in a relational category $\rr$. Then $e$ is an epi in $\rr$ iff $e$ is an extremal epi in $\ladj(\rr)$.
\end{corollary}
\begin{proof}
  The forward direction was shown in \cref{prop.extremal_epis}.
  
  For the converse, let $e$ be an extremal epi in $\ladj(\rr)$. Note that by \cref{prop.characterize_monos} and \cref{prop.adj_vs_hom}, $e$ is an epi in $\rr$ iff $e$ is a unit homomorphism. Thus use \cref{prop.factorization,prop.adj_vs_hom} to factor $e$ into a unit-homomorphic left adjoint followed by a mono; by definition this mono must be an iso. Hence $e$ is the composite of two unit homomorphisms, and hence itself a unit homomorphism.
\end{proof}

\begin{remark}
  Although we shall not need it, it is also true that a left adjoint $m$ in a relational po-category $\rr$ is mono in $\rr$ iff it is mono in $\ladj(\rr)$. Thus we see that extremal epis and monos in $\ladj(\rr)$ are the shadows of epis and monos in $\rr$.

  Note that these facts do not hold for prerelational po-categories; for example, they do not hold true in $\ww$; see \cref{ex.ww_not_relational}. 
\end{remark}

\begin{proposition}[Universal property of tabulation] \label{prop.univ_ppty_tabulation}
  Let $y\colon r_1 \to r_2$ have tabulation $y_R \cp y_L$. Suppose given an object $s$ and left adjoints $f_1\colon s\ladjto r_1$ and $f_2\colon s\ladjto r_2$ such that 
  \begin{equation}\label{eqn.ineq8493}
	\begin{tikzpicture}[unoriented WD, font=\small,baseline=(P1.-10)]
	  \node (P1) {
	    \begin{tikzpicture}[inner WD]
	      \node[funcl] (f) {$f_1$};
	      \node[funcr, right=.75 of f] (g) {$f_2$};
	      \draw (f) -- +(-10pt, 0);
	      \draw (f) -- (g);
	      \draw (g) -- +(10pt,0);
	    \end{tikzpicture}	
	  };
	  \node[right=2 of P1] (P2) {
	    \begin{tikzpicture}[inner WD]
	      \node[oshellr, inner sep=1pt] (f) {$y$};
	      \draw (f.west) -- +(-.5, 0);
	      \draw (f.east) -- +(.5,0);
	    \end{tikzpicture}	
	  };
	  \node at ($(P1.east)!.5!(P2.west)$) {$\le$};
  \end{tikzpicture}.
  \end{equation}
  Define $\hat{f}\colon s\to r_1\otimes r_2$, \;$\hat{y}\colon |y|\to r_1\otimes r_2$, and $h\colon s\to |y|$ to be the following morphisms:
  \[
  \begin{tikzpicture}[unoriented WD, font=\small,baseline=(P1.-10)]
  	\node (P-1) {
    \begin{tikzpicture}[inner WD, baseline=(dot1)]
    	\node[funcr] (fl) {$f_1$};
  		\node[funcr, below=.5 of fl] (fr!) {$f_2$};
  		\node[link] at ($(fl.west)!.5!(fr!.west)+(-1,0)$) (dot1) {};
  		\draw (fl.west) to[out=180, in=60] (dot1);
  		\draw (fr!.west) to[out=180, in=-60] (dot1);
			\draw (fl.east) -- +(.5,0);
			\draw (fr!.east) -- +(.5,0);
  		\draw (dot1) to +(-1,0) node[left=.3, font=\normalsize] {$\hat{f}\coloneqq$};
    \end{tikzpicture}
    };
  	\node (P0) [right=3 of P-1] {
    \begin{tikzpicture}[inner WD, baseline=(dot1)]
    	\node[funcr] (fl) {$y_R$};
  		\node[funcr, below=.5 of fl] (fr!) {$y_L$};
  		\node[link] at ($(fl.west)!.5!(fr!.west)+(-1,0)$) (dot1) {};
  		\draw (fl.west) to[out=180, in=60] (dot1);
  		\draw (fr!.west) to[out=180, in=-60] (dot1);
			\draw (fl.east) -- +(.5,0);
			\draw (fr!.east) -- +(.5,0);
  		\draw (dot1) to +(-1,0) node[left=.3, font=\normalsize] {$\hat{y}\coloneqq$};
    \end{tikzpicture}
    };
  	\node (P1) [right=3 of P0] {
    \begin{tikzpicture}[inner WD, baseline=(dot1)]
      \node[funcr] (f) {$h$};
      \draw (f) -- +(-15pt,0);
      \draw (f) -- +(15pt,0);
    \end{tikzpicture}
    };
  	\node[right=1 of P1] (P2) {
    \begin{tikzpicture}[inner WD, baseline=(dot1)]
    	\node[funcr] (fl) {$f_1$};
  		\node[funcl, right=.8 of fl] (fl*) {$y_R$};
  		\node[funcr, below=.5 of fl] (fr!) {$f_2$};
  		\node[funcl] at (fl*|-fr!) (fr) {$y_L$};
  		\draw (fl.east) -- (fl*.west);
  		\draw (fr!.east) -- (fr.west);
  		\node[link] at ($(fl.west)!.5!(fr!.west)+(-1,0)$) (dot1) {};
  		\node[link] at ($(fl*.east)!.5!(fr.east)+(1,0)$) (dot2) {};
  		\draw (fl.west) to[out=180, in=60] (dot1);
  		\draw (fr!.west) to[out=180, in=-60] (dot1);
  		\draw (dot1) to +(-1,0);
  		\draw (fl*.east) to[out=0, in=120] (dot2);
  		\draw (fr.east) to[out=0, in=-120] (dot2);
  		\draw (dot2) to +(1,0);
    \end{tikzpicture}
    };
    \node at ($(P1.east)!.5!(P2.west)$) {$\coloneqq$};
\end{tikzpicture}
\]  
  Then $h$ is a left adjoint, and it is the unique left adjoint obeying $\hat{f}=h\cp\hat{y}$.
  Finally, if \eqref{eqn.ineq8493} is an equality then $\hat{f}=h\cp\hat{y}$ is an image factorization (i.e.\ $h$ is extremal epi and $\hat{y}$ is mono).
\end{proposition}
\begin{proof}
  Letting $x\coloneqq (\eta_{|y|}\cp\hat{y})$, it follows from \eqref{eqn.ineq8493} that $(\eta_s\cp\hat{f})\leq x$. By \cref{prop.condition3}, $\hat{y}$ is monic so $(\eta_{|y|},\hat{y})$ is a tabulation of $x$. Thus by \cref{prop.restrict_codomain}, $h=\hat{f}\cp\hat{y}\tp$ is a left adjoint and $\hat{f}=h\cp\hat{y}$. Furthermore, if \eqref{eqn.ineq8493} is an equality then $h$ is an extremal epi. Uniqueness of $h$ follows from the fact that $\hat{y}$ is monic: if also $\hat{f}=h'\cp\hat{y}$ then $h=h'$.
 \end{proof}

\begin{corollary}\label{cor.tabulations_unique}
In a relational po-category, tabulations are unique up to isomorphism.
\end{corollary}
\begin{proof}
Suppose a map $y\colon r\to s$ has tabulations $f_R\cp f_L$ and $g_R\cp g_L$. By \cref{prop.univ_ppty_tabulation} we obtain universal left adjoints $h\colon|f|\to|g|$ and $h'\colon|g|\to|f|$, and they are mutually isomorphic.
\end{proof}

\begin{corollary} \label{cor.pullbacks}
  Let $g_1\colon r_1 \ladjto t$ and $g_2 \colon r_2 \ladjto t$ be left adjoints in $\rr$. Then their pullback in $\ladj(\rr)$ is given by the tabulation of $g_1 \cp g_2\tp$.
\end{corollary}
\begin{proof}
  This is exactly the content of \cref{prop.univ_ppty_tabulation} when $y\coloneqq g_1 \cp g_2\tp$. 
\end{proof}

\begin{proposition} \label{prop.stable_epis}
  In $\ladj(\rr)$, extremal epis are stable under pullback.
\end{proposition}
\begin{proof}
  Suppose $e$ is an extremal epi in $\ladj(\rr)$, and that pulling back along $f$ gives left adjoints $e'$ and $f'$ such that $f' \cp e = e' \cp f$. We simply use the facts that left adjoints are counit homomorphisms, and that for left adjoints, the extremal epis are precisely the unit homomorphisms (\cref{cor.extremal_epis}):
  \[
  \begin{tikzpicture}[baseline=(p1)]
  	\node (p1) {
		\begin{tikzpicture}[inner WD]
			\node[funcr] (m) {$e'$};
			\node[link, left=.5 of m] (link) {};
			\draw (m.west) -- (link);
			\draw (m) -- +(10pt,0);
		\end{tikzpicture}
		};
		\node (p2) [right=.5 of p1] {
		\begin{tikzpicture}[inner WD]
			\node[funcl] (f1) {$f'$};
			\node[funcr, right=.75 of f1] (f2) {$e'$};
			\node[link, left=.5 of f1] (link) {};
      \draw (f2) -- +(10pt,0);
      \draw (f1) -- (f2);
			\draw (f1.west) -- (link);
		\end{tikzpicture}		
		};
		\node (p3) [right=.5 of p2] {
		\begin{tikzpicture}[inner WD]
			\node[funcr] (f1) {$e$};
			\node[funcl, right=.75 of f1] (f2) {$f$};
			\node[link, left=.5 of f1] (link) {};
			\draw (f2) -- +(10pt,0);
      \draw (f1) -- (f2);
			\draw (f1.west) -- (link);
		\end{tikzpicture}		
		};
		\node (p4) [right=.5 of p3] {
		\begin{tikzpicture}[inner WD]
			\node[funcl] (m) {$f$};
			\node[link, left=.5 of m] (link) {};
			\draw (m.west) -- (link);
			\draw (m) -- +(10pt,0);
		\end{tikzpicture}		
		};
		\node (p5) [right=.5 of p4] {
		\begin{tikzpicture}[inner WD]
			\node[link] (link) {};
			\draw (link) -- +(10pt,0);
		\end{tikzpicture}		
		};
    \node at ($(p1.east)!.5!(p2.west)$) {$=$};
		\node at ($(p2.east)!.5!(p3.west)$) {$=$};
		\node at ($(p3.east)!.5!(p4.west)$) {$=$};
		\node at ($(p4.east)!.5!(p5.west)$) {$=$};
  \end{tikzpicture}
  \qedhere
  \]
\end{proof}

\begin{proposition}\label{prop.ladj_rel_reg}
  If $\rr$ is a relational po-category then $\ladj(\rr)$ is a regular category.
\end{proposition}
\begin{proof}
  The monoidal unit in $\rr$ is terminal in $\ladj(\rr)$ because every morphism is a counit homomorphism, and $\ladj(\rr)$ has pullbacks by \cref{cor.pullbacks}; thus $\ladj(\rr)$ has finite limits. Every morphism factors into an extremal epi followed by a mono by \cref{prop.factorization}, and extremal epis are stable under pullback by \cref{prop.stable_epis}. Thus $\ladj(\rr)$ is regular.
\end{proof}

\begin{remark}
  By \cref{prop.fox} one sees that the finite products in $\ladj(\rr)$ are given by the monoidal structure $I,\otimes$ in $\rr$. In particular for any object $r$, the unique map $r \to I$ is given by $\epsilon_r$ and similarly the projection map $r\otimes s\to r$ is given by $\id_r\otimes\epsilon_s$. The diagonal is given by the comultiplication $\delta$.
\end{remark}

\section{$\ladj$ is functorial}\label{sec.extend_ladj_fun}

We showed in \cref{prop.ladj_rel_reg} that the category of left adjoints in a relational po-category $\rr$ is regular. We now show this extends to a 2-functor $\ladj\colon \rrlcat \to \rrgcat$.

Recall from \cref{def.reg2cat} that we defined a morphism $F\colon\rr\to\rr'$ of relational po-categories simply to be a strong monoidal po-functor. One might suspect that a morphism should also be required to preserve the supply of wirings (see \cref{def.preserve_supply}) as well as the tabulations. We are now in a position to show that this is automatic.

\begin{proposition} \label{prop.relational_pofunctors}
If $\rr$ and $\rr'$ are prerelational po-categories and $F\colon\rr\to\rr'$ is any strong symmetric monoidal functor, then $F$ automatically preserves the supply of $\ww$. If moreover $\rr$ and $\rr'$ are relational then $F$ preserves tabulations.
\end{proposition}
\begin{proof}
To see $F$ preserves the supply of $\ww$, let $\varphi$ denote the strongator isomorphisms for $F$. By \cref{eqn.generating_wires,eqn.unpack_preserve_supply}, it suffices to show that the following diagrams commute for each $r\in\rr$:
\[
	\begin{tikzcd}[column sep=24pt]
		I\ar[r, "\eta_{F(r)}"]\ar[from=d, "\varphi\inv"]&
		F(r)\ar[d, equal]\\
		F(I)\ar[r, "F(\eta_r)"']&
		F(r)
  \end{tikzcd}
\hspace{.25in}
	\begin{tikzcd}[column sep=24pt]
		F(r)\tpow{2}\ar[r, "\mu_{F(r)}"]\ar[from=d, "\varphi\inv"]&
		F(r)\ar[d, equal]\\
		F(r\tpow{2})\ar[r, "F(\mu_r)"']&
		F(r)
  \end{tikzcd}
\hspace{.25in}
	\begin{tikzcd}[column sep=24pt]
		F(r)\ar[d, equal]\ar[r, "\epsilon_{F(r)}"]&
		I\ar[from=d, "\varphi\inv"']\\
		F(r)\ar[r, "F(\epsilon_r)"']&
		F(I)
  \end{tikzcd}
\hspace{.25in}
	\begin{tikzcd}[column sep=24pt]
		F(r)\ar[d, equal]\ar[r, "\delta_{F(r)}"]&
		F(r)\tpow{2}\ar[from=d, "\varphi\inv"']\\
		F(r)\ar[r, "F(\delta_r)"']&
		F(r\tpow{2})
  \end{tikzcd}
\]
Since $\delta$ and $\epsilon$ are left adjoint to $\mu$ and $\eta$, and $F$ preserves adjunctions, it suffices to show the commutativity of either pair. Since isomorphisms are left adjoints, $\epsilon\coloneqq\varphi\inv\cp F(\epsilon_r)$ and $\delta\coloneqq\varphi\inv\cp\delta_{F(r)}$ are left adjoints, and $(\epsilon,\delta)$ also form a comonoid structure on $r$ in the usual way (strong monoidal functors send comonoids to comonoids). Hence we are done by the uniqueness of left adjoint comonoids (\cref{prop.prerel_comon_unique}).

The second claim is obvious: tabulations are defined equationally, so $F$ must preserve them.
\end{proof}

\begin{proposition}\label{prop.ladj_morphisms}
If $F\colon\rr\to\ccat{S}$ is a relational po-functor, then $\ladj(F)\colon\ladj(\rr)\to\ladj(\ccat{S})$ is a regular functor.
\end{proposition}
\begin{proof}
  $\ladj(F)$ is a well-defined functor with codomain $\ladj(\ccat{S})$; see \cref{prop.ladj_smc}. To show it is regular, we must show it preserves finite limits and extremal epis. Since $F$ is strong and the monoidal units in $\rr$ and $\ccat{S}$ are terminal, $\ladj(F)$ sends terminal objects to terminal objects. \cref{cor.pullbacks} says that pullbacks are given by tabulations, and $F$ preserves tabulations by \cref{prop.relational_pofunctors}. Finally, \cref{cor.extremal_epis} shows that extremal epis are given by unit homomorphisms, and since $F$ preserves the supply by \cref{prop.relational_pofunctors}, it sends unit homomorphisms to unit homomorphisms.
\end{proof}

We recall that a 2-morphism in $\rrlcat$ is a left adjoint natural transformation $\alpha\colon F \Rightarrow G$ between strong monoidal functors $F,G \colon \rr \to \ccat{S}$. Explicitly (see \cref{prop.adj_in_pocat}) it is a natural transformation such that each component map $\alpha_r\colon Fr\to Gr$ is a left adjoint in $\ccat{S}$.

\begin{proposition}\label{prop.ladj_mmorphisms}
Every left adjoint natural transformation $\alpha\colon F\Rightarrow G$ induces a natural transformation $\alpha\colon \ladj(F)\Rightarrow\ladj(G)$ between regular functors $\ladj(\rr)\to\ladj(\ccat{S})$.
\end{proposition}
\begin{proof}
  This follows directly from \cref{prop.ladj_smc}.
\end{proof}

\begin{theorem}\label{thm.ladj}
There is a 2-functor $\ladj\colon\rrlcat\to\rrgcat$.
\end{theorem}
\begin{proof}
This 2-functor was given on objects, 1-morphisms, and 2-morphisms in \cref{prop.ladj_rel_reg,prop.ladj_morphisms,prop.ladj_mmorphisms} respectively.
\end{proof}

\chapter{$\mathbb{R}\Cat{el}$ and $\ladj$ form a 2-equivalence}\label{sec.prove_equiv}

In \cref{thm.rrel} we showed that there is a 2-functor $\rrel\colon\rrgcat\to\rrlcat$, and in \cref{thm.ladj} we showed that there is a 2-functor $\ladj$ in the opposite direction. Our goal in this chapter is to prove \cref{thm.eequivalence}, the equivalence of 2-categories
\[
\begin{tikzcd}[column sep=large]
	\rrgcat\ar[r, shift left, "\rrel"]&
	\rrlcat.\ar[l, shift left, "\ladj"]
\end{tikzcd}
\]

\begin{lemma}[Fundamental lemma of regular categories]\label{lemma.fundamental}
Let $\cat{R}$ be a regular category. Then there is a natural identity-on-objects isomorphism of categories
\[\iota_{\cat{R}}\colon\cat{R}\to\ladj(\rrel(\cat{R})).\]
In particular, a relation $x\colon r\tickar s$ is a left adjoint iff it is the graph $x=\pair{\id_r,f}$ of a morphism $f\colon r\to s$ in $\cat{R}$.
\end{lemma}
\begin{proof}
This fact is well known, but we provide a proof here since it plays a central role. We shall show that there is an
identity-on-objects, fully faithful functor from $\cat{R}$ to its relations
po-category $\rrel(\cat{R})$, which sends a morphism $f\colon r \to s$ to its
graph $\iota_{\cat{R}}(f)\coloneqq \pair{\id_r,f} \ss r \times s$. Indeed, the pair $\pair{\id_r,f} \dashv \pair{f,\id_s}$ is an adjunction by \cref{rem.rels_adjoint_composites}, and it follows that $\iota_{\cat{R}}$ is functorial. This construction is 1-natural in $\cat{R}$, because any regular functor $F\colon\cat{R}\to\cat{S}$ sends graphs to graphs; it is easy to check that it is also 2-natural.

To show that $\iota_{\cat{R}}$ is fully faithful, we characterize the adjunctions $x \dashv
x'$ in $\rrel(\cat{R})$. Suppose we have $\pair{g,f}\colon x \inj r \times
s$ and $\pair{f',g'}\colon x'\inj s\times r$ with unit $i\colon r\inj (x\cp x')$
and counit $j\colon (x'\cp x)\to s$. This gives rise to the following diagram
(equations shown right):
\[
\begin{tikzcd}[row sep=6pt, column sep=large]
	x\times_s x'\ar[rr, "\pi_s'"]\ar[ddd, "\pi_s"']&&
	x'\ar[dl, pos=.6, "g'"']\ar[ddl, "f'"]\\&
	r\ar[ul, "i"]\\&
	s\\
	x\ar[uur, "g"]\ar[ur, pos=.6, "f"']&&
	x\times_r x'\ar[uuu, "\pi_r'"']\ar[ll, "\pi_r"]\ar[ul, "j"]
\end{tikzcd}
\hspace{.5in}
\parbox{1.5in}{$
i\cp\pi_s\cp g=\id_r=i\cp\pi_s'\cp g'\\
\pi_r\cp f=j=\pi_r'\cp f'
$}
\]
We shall show that $g$ and $g'$ are isomorphisms, and that $f'=g' \cp g\inv \cp
f$. 

We first show that $i\cp \pi_s$ is inverse to $g$. Since the unit already gives
that $i \cp \pi_s \cp g = \id_r$, it suffices to show that $g\cp i\cp\pi_s=\id_x$.
Moreover, since $\pair{g,f}\colon x\to r\times s$ is monic and $g=(g\cp
i\cp\pi_s)\cp g$, it suffices to show that $f=(g\cp i\cp\pi_s)\cp f$. This is a diagram chase: since $g=g\cp i\cp\pi_s'\cp g'$, we can define a map $q\coloneqq\pair{\id_x,(g\cp i\cp\pi_s')}\colon x\to x\times_r x'$, and we conclude
\[
	f=q\cp \pi_r\cp f
	=q\cp \pi_r'\cp f'
	=g\cp i\cp\pi_s'\cp f'
	=g\cp i\cp\pi_s\cp f.
\]
Similarly, we see that $i\cp \pi_s'$ is inverse to $g'$, and hence obtain $f'=g' \cp g\inv \cp f$.

Note that this implies the adjunction $x \dashv x'$ is isomorphic to the
adjunction $\pair{\id_r,(g\inv \cp f)} \dashv \pair{(g \inv \cp f),\id_s}$. Thus the
proposed functor is full. Faithfulness amounts to the fact that the existence of
a morphism $\pair{\id_r,f} \to \pair{\id_r,f'}$ implies $f =f'$. This proves the
lemma.
\end{proof}

We showed in \cref{lemma.fundamental} that there is a natural isomorphism $\iota\colon\id_{\rrgcat}\to(\ladj\circ\rrel)$. We need to show the other round-trip.

\begin{proposition}\label{lemma.fundamental_po}
For any relational po-category $\rr$, there is a natural identity-on-objects
isomorphism of monoidal po-categories $j_\rr\colon\rr\cong\rrel(\ladj(\rr))$.
\end{proposition}
\begin{proof}
The po-categories $\rr$ and $\rrel(\ladj(\rr))$ have the same set of objects and monoidal product by construction. On 1-morphisms, a tabulation of a morphism $f\colon r \to s$ in $\rr$ defines a jointly-monic span $\hat{f}$ of left adjoints as shown in \cref{prop.condition3}; it is unique up to isomorphism by \cref{cor.tabulations_unique} and hence defines a morphism $j_R(f)\colon r \tickar s$ in $\rrel(\ladj(\rr))$. Our first goal is to show that $j_R$ is functorial.

Given morphisms $f\colon r\to s$ and $g\colon s\to t$ in $\rr$, let $y=f\cp g$ be their composite, and let $(f_R, f_L)$ and $(g_R,g_L)$ be their tabulations. They form a composable pair of relations:
\begin{equation}\label{eqn.74}
\begin{tikzcd}
	r&
	{|f|}\ar[l, "f_R"']\ar[r, "f_L"]&
	s&
	{|g|}\ar[l, "g_R\tp"']\ar[r, "g_L"]&
	t
\end{tikzcd}
\end{equation}
and their composite in $\rrel(\ladj(\rr))$ is given by taking the pullback of the middle terms, composing with $f_R$ and $g_L$, and taking the image inclusion of the resulting span. By \cref{cor.pullbacks}, the pullback of middle terms is the tabulation $(x_R,x_L)$ of $f_L\cp g_R\tp$. Composing with $f_R$ and $g_L$, we have simply replaced \cref{eqn.74} by a chain of morphisms
\[
\begin{tikzcd}
	r&
	{|f|}\ar[l, "f_R"']\ar[from=r, "x_R"']&
	{|x|}&
	{|g|}\ar[from=l, "x_L\tp"]\ar[r, "g_R"]&
	t
\end{tikzcd}
\]
with the same composite, $y=(f_R\cp x_R)\cp (x_L\cp g_L)$. We want the image inclusion for the associated span to be $\hat{y}$, and it is by \cref{prop.univ_ppty_tabulation}. Note that $j$ is natural since functors $\rr\to\rr'$ preserve transposes and tabulations (see \cref{prop.relational_pofunctors}).

To define an identity-on-objects po-functor $j_\rr\inv\colon\rrel(\ladj(\rr))\to\rr$, we just need to say what it does on hom-posets. A relation $f\colon r_1 \tickar r_2$ in $\rrel(\ladj(\rr))$ is a jointly monic span $\hat{f}$ of left adjoints $f_1\colon s\to r_1$ and $f_2\colon s\to r_2$ for some $s\in\rr$. We send $f$ to the composite ${f_1}\tp \cp f_2$ in $\rr$. If $f\leq f'$ then \cref{prop.univ_ppty_tabulation} defines a map $h\colon s\to s'$ with $\hat{f}=h\cp\hat{f'}$, so this construction is monotonic.

It is clear that $j_\rr$ and $j_\rr\inv$ really are mutually inverse, completing the proof.
\end{proof}

\begin{theorem}\label{thm.eequivalence}
The 2-functors $\rrel\colon\rrgcat\leftrightarrows\rrlcat\cocolon\ladj$
form an equivalence of 2-categories. Their underlying 1-functors moreover form an equivalence between the underlying 1-categories.
\end{theorem}
\begin{proof}
This is just \cref{lemma.fundamental,lemma.fundamental_po}. The second statement says that, for every relational category $\rr$, we have not just an equivalence but an isomorphism of po-categories $\rr\cong\rrel(\ladj(\rr))$. Similarly, for every regular category $\cat{R}$, we have not just an equivalence but an isomorphism of categories $\cat{R}\cong\ladj(\rrel(\rr))$.
\end{proof}

\chapter{Relational po-categories, cartesian bicategories, allegories} \label{chap.carboni_walters}

We close by comparing relational po-categories to two well known axiomizations of bicategories of relations: cartesian bicategories and allegories.

\section{Cartesian bicategories}

A prerelational po-category is exactly a `bicategory of relations'.

\begin{definition}[`Bicategory of relations' {\cite[Definition~2.1]{carboni1987cartesian}}] \label{def.bicategory_of_relations}
  Let $\cc$ be a symmetric monoidal po-category. We say that $\cc$ is a \emph{`bicategory of relations'} if for every object $c$ in $\cc$ we have maps $
  \delta_c
  \colon c \to c \otimes c$ and  
  $\epsilon_c
  \colon c \to I$ such that:
  \begin{enumerate}[label=(\roman*)]
    \item $(c,\delta_c,\epsilon_c)$ forms a commutative comonoid.
    \item $\delta_c$ and $\epsilon_c$ have right adjoints 
    $\delta_c^*
    $
    and 
    $\epsilon_c^*
    $.
    \item $\delta_c$ and $\delta_c^*$ obey the frobenius law: $\delta_c^* \cp \delta_c = (1 \otimes \delta_c) \cp (\delta_c^* \otimes 1$). In pictures:
    \[
	\begin{tikzpicture}[WD]
    \node[link] (delta) {};
   	\draw (delta) -- +(-.5,0) coordinate (end);
    \draw (delta) to[out=60, in=180] +(1,.5);
   	\draw (delta) to[out=-60, in=180] +(1,-.5);
    \node[link, left=1 of delta] (mu) {};
   	\draw (mu) -- (end);
    \draw (mu) to[out=120, in=0] +(-1,.5);
   	\draw (mu) to[out=-120, in=0] +(-1,-.5);
	\end{tikzpicture}
\;\raisebox{2pt}{=}\quad
  \begin{tikzpicture}[WD]
  	\coordinate (htop);
    \coordinate[below=.7 of htop] (hmid);
    \coordinate[below=.7 of hmid] (hbot);
    \node[link, left=.5] (dotL) at ($(hbot)!.5!(hmid)$) {};
    \node[link, right=.5] (dotR) at ($(htop)!.5!(hmid)$) {};
    \draw (dotL) -- +(-1,0) coordinate (l);
    \draw (dotR) -- +(1,0) coordinate (r);
    \draw (dotL) to[out=-60, in=180] (hbot);
    \draw (dotL) to[out=60, in=180] (hmid);
    \draw (hmid) to[out=0, in=-120] (dotR);
    \draw (htop) to[out=0, in=120] (dotR);
    \draw (htop) -- (htop-|l);
    \draw (hbot) -- (hbot-|r);
  \end{tikzpicture}
    \]
    \item each 1-morphism $f\colon c \to d$ is a lax comonoid homomorphism.
    \item $(\delta_c,\epsilon_c)$ is the unique comonoid structure on $c$ with these properties.
  \end{enumerate}
\end{definition}

The key difference in perspective is that while we have chosen to supply our category with wirings $\ww$ (the local poset reflection of $\ccospan\co$), Carboni and Walters give a more syntactic definition, explicitly giving a diagonal map $\delta$ and a terminal map $\epsilon$ for each object and requiring them to obey various conditions.

Note that a supply of wirings already implies we have $\delta_c$ and $\epsilon_c$ obeying (i), (ii), and (iii); see \cref{eqn.generating_wires,eqn.equations_wires}. To show the two definitions are the same, we thus just need to show that (i), (ii), and (iii) amount to a supply of wirings. To do this, we shall prove that \cref{eqn.generating_wires,eqn.equations_wires} present $\ww$. 

\begin{proposition}[Wirings are the po-prop for adjoint frobenius monoids] \label{prop.lafms_are_relfinsetop}
	The generators and relations shown in \cref{eqn.generating_wires,eqn.equations_wires} are complete for the po-prop $\ww$.
\end{proposition}
\begin{proof}
Let $\ccat{X}$ denote the prop presented by \cref{eqn.generating_wires,eqn.equations_wires}, and recall that $\ww$ is the local poset reflection of $\ccospan\co$. Define a prop functor $\ccat{X}\to\ww$ by sending $\epsilon$ and $\delta$ to the unique cospans of the form $0\to 1\from 1$ and $2\to 1\from 1$, and sending $\eta$ and $\mu$ to their respective transposes, which happen to be their left adjoints. It is easy to check that the equations \cref{eqn.equations_wires} hold in $\ww$, and so this is well defined.

We claim this functor is fully faithful and locally full (it is obviously locally faithful). Fullness was shown in \cite[Lemma 3.6]{fong2019hypergraph}, and we leave faithfulness to the reader. In the remainder we show that every 2-morphism in $\ww$ is a composite of tensors of the 2-morphisms $(\eta\cp\epsilon)\leq \id_0$ and $(\delta\cp\mu)\leq \id_2$.

Let $m\To{f}p\From{g}n$ be any 1-morphism in $\ww$, and consider the epi-mono factorization $m+n\surj i\inj p$ of the copairing $\copair{f}{g}\colon m+n\to p$. This provides a 2-morphism $i\tto p$ in $\ccospan$. As long as $i\neq 0$, there is also a 2-morphism $p\tto i$, because every mono between nonempty finite sets has a retraction; thus in the poset reflection $\ww$ of $\ccospan\co$ we have $i=p$. However, if $i=0$ then $p\leq i$ is a strict inequality in $\ww$. We can only have $i=0$ if $m=n=0$, and we see that $i$ is just the image of $(\eta\cp\epsilon)$ under our functor. Thus we may assume that all our cospans are jointly epic. It follows that any morphism $q\tto p$ of cospans is also epic.

Suppose given a 2-morphism $p\tto q$ in $\ww(m,n)$, i.e.\ a morphism $q\tto p$ of jointly epic cospans. It can be factored as follows:
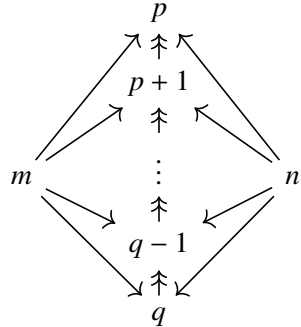
\[
\begin{tikzcd}[row sep=10pt]
  &
  p\\&
  p+1\ar[u, ->>]\\
  m\ar[uur]\ar[ur]\ar[dr]\ar[ddr]&
  \vdots\ar[u, ->>]&
  n\ar[uul]\ar[ul]\ar[dl]\ar[ddl]\\&
  q-1\ar[u, ->>]\\&
  q\ar[u, ->>]
\end{tikzcd}
\]
where we are using arrows to represent functions (morphisms in $\finset$). It suffices to assume $q=p+1$. Moreover, this map is the tensor product of $p-1$ identity 2-morphisms and one 2-morphism of the form shown left
\[
\begin{tikzcd}[row sep=0pt]
	&
	1\\
	m'\ar[ur]\ar[dr]&&
	n'\ar[ul]\ar[dl]\\&
	2\ar[uu]
\end{tikzcd}
\quad=\quad
\begin{tikzcd}[row sep=0pt]
	&
	2&&
	1&&
	2\\
	m'\ar[ur]\ar[dr]&&
	2\ar[ul, equal]\ar[dl, equal]\ar[dr, equal]\ar[ur]&&
	2\ar[ul]\ar[dl, equal]\ar[ur, equal]\ar[dr, equal]&&
	n'\ar[ul]\ar[dl]\\&
	2\ar[uu, equal]&&
	2\ar[uu]&&
	2\ar[uu, equal]
\end{tikzcd}
\]
but any such 2-morphism can be factored as shown right. The middle map is the 2-morphism $(\delta\cp\mu)\leq\id_2$, as desired.
\end{proof}

\begin{proposition}
  A monoidal po-category is a prerelational po-category iff it is a `bicategory of relations'.
\end{proposition}
\begin{proof}
  Note that any prerelational po-category is a `bicategory of relations', with the comonoid $(\delta_c,\epsilon_c)$ given by the supply of wirings, and the uniqueness (iv) given by \cref{prop.prerel_comon_unique}.

  For the converse, it is enough to show that (i), (ii), (iii), and (v) in \cref{def.bicategory_of_relations} implies a supply of wirings.   To see that each object $c$ is equipped with a wiring, we use the presentation \cref{eqn.equations_wires} as per \cref{prop.lafms_are_relfinsetop}; we define of course $\delta = \delta_c$, $\epsilon = \epsilon_c$, $\mu = \delta_c^*$, and $\eta = \epsilon_c^*$. Condition (i) states that $(\delta,\epsilon)$ form a commutative comonoid, and so their adjoints $(\mu,\eta)$ form a commutative monoid. Moreover, Condition (ii) takes care of the adjointness equations. It is well known the pairwise frobenius law of Condition (iii) is equivalent to the three part version in \cref{eqn.equations_wires}.\footnote{See for example: https://graphicallinearalgebra.net/2015/10/01/22-the-frobenius-equation/.}
 The special law is derivable; see \cref{rem.surprising_inequality}. Finally, the fact that these wirings form a supply follows from the uniqueness condition (v).
\end{proof}

A relational po-category is exactly a functionally complete `bicategory of relations'.

\begin{definition}[Functional completeness {\cite[Definition 3.1]{carboni1987cartesian}}]
  A `bicategory of relations' is \emph{functionally complete} if every 1-morphism $r \to I$ has a tabulation.
\end{definition}

\begin{proposition}
  A monoidal po-category is a relational po-category iff it is a functionally complete `bicategory of relations'. 
\end{proposition}
\begin{proof}
  We simply need to check that tabulations for morphisms $r \to I$ imply tabulations for all morphisms. This is the content of \cite[Corollary~3.4(i)]{carboni1987cartesian}.
\end{proof}

\section{Allegories}

In this final section we outline the connection with Freyd's notion of allegory. A full treatment of allegories and their relationship with regular categories can be found in \cite[Chapter A3]{Johnstone:2002a}.

\begin{definition}[Unital allegory] \label{def.allegory}
  A po-category $\cc$ is an \emph{allegory} if it is equipped with an identity-on-objects po-functor $\tp\colon \cc\op \to \cc$ such that 
  \begin{enumerate}[label=(\roman*)]
    \item $\tp \cp \tp = \id$.
    \item each hom-poset $\cc(x,y)$ has binary meets.
    \item for all $f\colon x \to y$, $g\colon y \to z$, and $h\colon x \to z$ the \emph{modular law} holds: $(f \cp g) \wedge h \le f \cp (g \wedge (f\tp \cp h))$.
  \end{enumerate}

  We further say that an allegory is \emph{unital} if there exists an object $u$ such that
  \begin{enumerate}[label=(\roman*)]
    \item[(iv)] $\id_u$ is the top element of $\cc(u,u)$.
    \item[(v)] every object $x$ there is a morphism $\epsilon\colon x \to u$ such that $\epsilon \cp \epsilon\tp \ge \id_x$.
  \end{enumerate}
  We call any such $u$ a \emph{unit}.
\end{definition}

\begin{definition}[tabular allegory]
  A \emph{tabulation} of a morphism $f$ in an allegory is a pair of morphisms $f_R,f_R$ such that $f = f_R\cp f_L$ and $(f_L \cp {f_L}\tp) \wedge (f_R \cp {f_R}\tp) = \id$. We say that an allegory is \emph{tabular} if every morphism has a tabulation.
\end{definition}

It is perhaps instructive to note that any prerelational po-category is a unital allegory. Indeed, given a prerelational po-category, the self-duality on objects provides the identity-on-objects involution $\tp$; this implies Condition (i) of \cref{def.allegory}.

We have previously mentioned that the hom-posets of a prerelational po-category are meet-semilattices, which clearly implies Condition (ii); this can be proved as follows.

\begin{proposition} \label{prop.meetsl}
  In any prerelational po-category, the hom-posets are meet-semilattices with join of $f$ and $g$ given by 
  $
    \begin{tikzpicture}[inner WD,inner sep=1pt,baseline=(dot)]
    	\node[oshellr] (f1) {$f$};
  		\node[oshellr, below=.25 of f1] (f2) {$g$};
  		\node[link] at ($(f1.west)!.5!(f2.west)+(-1,0)$) (dot) {};
  		\node[link] at ($(f1.east)!.5!(f2.east)+(1,0)$) (dot2) {};
  		\draw (f1.west) to[out=180, in=60] (dot);
  		\draw (f2.west) to[out=180, in=-60] (dot);
  		\draw (dot) to +(-1,0);
  		\draw (f1.east) to[out=0, in=120] (dot2);
  		\draw (f2.east) to[out=0, in=-120] (dot2);
  		\draw (dot2) to +(1,0);
    \end{tikzpicture}
  $
  and top element given by 
  $
  \begin{tikzpicture}[inner WD]
    \node[link] (dot1) {};
    \node[link, right=.5 of dot1] (dot2) {};
    \draw (dot1) -- +(-1,0);
    \draw (dot2) -- +(1,0);
  \end{tikzpicture}
  $\;.
\end{proposition}
\begin{proof}
For any $c,c'\in\cc$, we show that $X\coloneqq\cc(c,c')$ is a meet-semilattice. To see that
$
  \begin{tikzpicture}[inner WD]
    \node[link] (dot1) {};
    \node[link, right=.5 of dot1] (dot2) {};
    \draw (dot1) -- +(-1,0);
    \draw (dot2) -- +(1,0);
  \end{tikzpicture}
$
is the top element, consider the following chain of inequalities for any $x\in X$:
\[
\begin{tikzpicture}
	\node (p1) {
  \begin{tikzpicture}[inner WD]
		\node[oshellr] (x) {$x$};
		\draw (x.west) -- +(-.5,0);
		\draw (x.east) -- +(1,0);
  \end{tikzpicture}	
	};
	\node (p2) [right=.5 of p1] {
  \begin{tikzpicture}[inner WD]
  	\node[oshellr] (x) {$x$};
    \node[link, right=.5 of x] (dot1) {};
    \node[link, right=.5 of dot1] (dot2) {};
    \draw (dot1) -- (x);
		\draw (x.west) -- +(-.5,0);
    \draw (dot2) -- +(1,0);
  \end{tikzpicture}
  };
	\node (p3) [right=.5 of p2] {
  \begin{tikzpicture}[inner WD]
    \node[link] (dot1) {};
    \node[link, right=.5 of dot1] (dot2) {};
    \draw (dot1) -- +(-1,0);
    \draw (dot2) -- +(1,0);
  \end{tikzpicture}
  };
	\node at ($(p1.east)!.5!(p2.west)$) {$\leq$};
	\node at ($(p2.east)!.5!(p3.west)$) {$\leq$};
 \end{tikzpicture}
\vspace{-.3in}\]
following from \cref{eqn.equations_wires,eqn.lax_hom_prerel}. Let $f\wedge g$ denote the morphism $
    \begin{tikzpicture}[inner WD,inner sep=1pt,baseline=(f2.south)]
    	\node[oshellr] (f1) {$f$};
  		\node[oshellr, below=.25 of f1] (f2) {$g$};
  		\node[link] at ($(f1.west)!.5!(f2.west)+(-1,0)$) (dot) {};
  		\node[link] at ($(f1.east)!.5!(f2.east)+(1,0)$) (dot2) {};
  		\draw (f1.west) to[out=180, in=60] (dot);
  		\draw (f2.west) to[out=180, in=-60] (dot);
  		\draw (dot) to +(-1,0);
  		\draw (f1.east) to[out=0, in=120] (dot2);
  		\draw (f2.east) to[out=0, in=-120] (dot2);
  		\draw (dot2) to +(1,0);
    \end{tikzpicture}
  $; we want to show it really is a meet. 
  
  By the above work, $f\wedge  g\leq f$ and $f\wedge g\leq g$, using the unitality for monoids and comonoids. To prove it has the universal property, suppose $h\leq f$ and $h\leq g$. Then $h\wedge h\leq f\wedge g$, so it suffices to show $h\leq h\wedge h$. This follows from the special law \eqref{eqn.equations_wires} and the fact that $h$ is a lax comonoid homomorphism and an oplax monoid homomorphism (see \cref{lemma.lax_comons_oplax_mons}).
\end{proof}

Finally, the modular law (Condition (iii) of \cref{def.allegory}), follows immediately from the fact that morphisms in a prerelational po-category are lax comonoid homomorphisms. 

\begin{proposition}
  Let $f\colon x \to y$, $g\colon y \to z$, and $h\colon x \to z$ be morphisms in a prerelational po-category. Then the modular law holds:  $(f \cp g) \wedge h \le f \cp (g \wedge (f\tp \cp h))$.
\end{proposition}
\begin{proof}
We want to prove the following inequality:
  
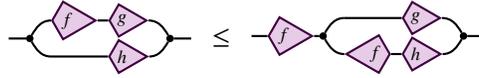
\[
    \begin{tikzpicture}
      \node (P1) {
        \begin{tikzpicture}[inner WD]
          \node[oshellr] (f1) {$f$};
          \node[oshellr, right=.5 of f1] (f2) {$g$};
          \node[oshellr, below=.25 of f2] (f3) {$h$};
          \coordinate (botleft) at (f3-|f1.west);
          \node[link] at ($(f1.west)!.5!(botleft)+(-1,0)$) (dot) {};
          \node[link] at ($(f2.east)!.5!(f3.east)+(1,0)$) (dot2) {};
          \draw (f1.west) to[out=180, in=60] (dot);
          \draw (botleft) to[out=180, in=-60] (dot);
          \draw (dot) to +(-1,0);
          \draw (f1.east) to (f2.west);
          \draw (f2.east) to[out=0, in=120] (dot2);
          \draw (botleft) -- (f3.west);
          \draw (f3.east) to[out=0, in=-120] (dot2);
          \draw (dot2) to +(1,0);
        \end{tikzpicture}
      };
      \node (P2) [right=.5 of P1] {
        \begin{tikzpicture}[inner WD]
          \node[oshelll] (f1) {$f$};
          \node[oshellr, right=.5 of f1] (f3) {$h$};
          \node[oshellr, above=.25 of f3] (f2) {$g$};
          \coordinate (topleft) at (f2-|f1.west);
          \node[link] at ($(f1.west)!.5!(topleft)+(-1,0)$) (dot) {};
          \node[link] at ($(f2.east)!.5!(f3.east)+(1,0)$) (dot2) {};
          \node[oshellr, left=.25 of dot] (f4) {$f$};
          \draw (topleft) to[out=180, in=60] (dot);
          \draw (f1.west) to[out=180, in=-60] (dot);
          \draw (dot) to (f4);
          \draw (f4.west) to +(-.75,0);
          \draw (topleft) to (f2.west);
          \draw (f2.east) to[out=0, in=120] (dot2);
          \draw (f1.east) to (f3.west);
          \draw (f3.east) to[out=0, in=-120] (dot2);
          \draw (dot2) to +(1,0);
        \end{tikzpicture}
      };
      \node at ($(P1.east)!.5!(P2.west)$) {$\leq$};
    \end{tikzpicture}
  \]
Using the hypergraph structure, we can rearrange the left-hand junction in such a way that both $f$'s are on its left; see (\cref{rem.frob_means_freedom}). Then the required inequality is simply the fact that $f$ is a lax comonoid homomorphism. 
\end{proof}

\begin{remark}
The simplicity of this graphical proof, and the derivation of the modular law from the primitive wiring/frobenius operations, was emphasised informally by Walters as a key advantage of the `bicategory of relations' approach.\footnote{For example, here: http://rfcwalters.blogspot.com/2009/10/categorical-algebras-of-relations.html.}
We are inclined to agree.
\end{remark}

Thus every prerelational po-category is an allegory, in fact a unital allegory by taking $u$ to be the monoidal unit.

We close by mentioning that relational po-categories are unital tabular allegories.

\begin{proposition}
  Let $\cc$ be a monoidal po-category. Then $\cc$ is a relational po-category iff it is a unital tabular allegory.
\end{proposition}
\begin{proof}
 It is well-known (see \cite[Theorem~A.3.2.10]{Johnstone:2002a}) that unital tabular allegories have a regular category $\cat{R}$ of left adjoints and in fact are isomorphic to the relations po-category of $\cat{R}$. The result then follows from our main theorem (\ref{thm.eequivalence}).
\end{proof}
  

\printbibliography

\end{document}

%% file: tikz-stuff.tex
\newif\ifpgfshaperectangleroundnortheast
\newif\ifpgfshaperectangleroundnorthwest
\newif\ifpgfshaperectangleroundsoutheast
\newif\ifpgfshaperectangleroundsouthwest
\pgfkeys{/pgf/.cd,
  rectangle round north east/.is if=pgfshaperectangleroundnortheast,
  rectangle round north west/.is if=pgfshaperectangleroundnorthwest,
  rectangle round south east/.is if=pgfshaperectangleroundsoutheast,
  rectangle round south west/.is if=pgfshaperectangleroundsouthwest,
  rectangle round north east, rectangle round north west,
  rectangle round south east, rectangle round south west,
}
\makeatletter
\def\pgf@sh@bg@rectangle{%
  \pgfkeysgetvalue{/pgf/outer xsep}{\outerxsep}%
  \pgfkeysgetvalue{/pgf/outer ysep}{\outerysep}%
  \pgfpathmoveto{\pgfpointadd{\southwest}{\pgfpoint{\outerxsep}{\outerysep}}}%
  {\ifpgfshaperectangleroundnorthwest\else\pgfsetcornersarced{\pgfpointorigin}\fi%
    \pgfpathlineto{\pgfpointadd{\southwest\pgf@xa=\pgf@x\northeast\pgf@x=\pgf@xa}{\pgfpoint{\outerxsep}{-\outerysep}}}}%
  {\ifpgfshaperectangleroundnortheast\else\pgfsetcornersarced{\pgfpointorigin}\fi%
    \pgfpathlineto{\pgfpointadd{\northeast}{\pgfpoint{-\outerxsep}{-\outerysep}}}}%
  {\ifpgfshaperectangleroundsoutheast\else\pgfsetcornersarced{\pgfpointorigin}\fi%
    \pgfpathlineto{\pgfpointadd{\southwest\pgf@ya=\pgf@y\northeast\pgf@y=\pgf@ya}{\pgfpoint{-\outerxsep}{\outerysep}}}}%
  {\ifpgfshaperectangleroundsouthwest\else\pgfsetcornersarced{\pgfpointorigin}\fi%
    \pgfpathclose}}

\tikzset{
  	WD/.style={
  	  label/.style={
      	font=\everymath\expandafter{\the\everymath\scriptstyle},
        inner sep=0pt,
        node distance=2pt and -2pt},
      semithick,
      node distance=\bbx and \bby,
      decoration={markings, mark=at position \stringdecpos with \stringdec},
    	bb port length=0,
  	  bb port sep=.5,
	 	  bbx = .4cm,
		  bb min width=.4cm,
	    bby = 2ex,
	    bb rounded corners=2pt,
	    dot size=3pt,
      pack size = 16pt,
    	penetration = 0pt,
      link size = 2pt,
      pack color = blue,
      surround sep=2pt,
      ar/.style={postaction={decorate}},
  		execute at begin picture={\tikzset{
  			x=\bbx, y=\bby, 
				circuit logic US, tiny circuit symbols
				}
			}
    },
    bbx/.store in=\bbx,
    bby/.store in=\bby,
    bb port sep/.store in=\bbportsep,
    bb port length/.store in=\bbportlen,
    bb min width/.store in=\bbminwidth,
    bb rounded corners/.store in=\bbcorners,
    bb/.code 2 args={
    	\pgfmathsetlengthmacro{\bbheight}{\bbportsep * (max(#1,#2)+1) * \bby}
      \pgfkeysalso{draw,minimum height=\bbheight,minimum
       width=\bbminwidth,outer sep=0pt,
         rounded corners=\bbcorners,thick,
         prefix after command={\pgfextra{\let\fixname\tikzlastnode}},
         append after command={\pgfextra{\draw
            \ifnum #1=0{} \else foreach \i in {1,...,#1} {
            	($(\fixname.north west)!{(2*\i-1)/(2*#1)}!(\fixname.south west)$) +(-\bbportlen,0) coordinate (\fixname_in\i) -- +(\bbportlen,0) coordinate (\fixname_in\i')}\fi 
            \ifnum #2=0{} \else foreach \i in {1,...,#2} {
            	($(\fixname.north east)!{(2*\i-1)/(2*#2)}!(\fixname.south east)$) +(-
\bbportlen,0) coordinate (\fixname_out\i') -- +(\bbportlen,0) coordinate (\fixname_out\i)}\fi;
           }}}
		},
	dot size/.store in=\dotsize,
	dot/.style={
		circle, draw, thick, inner sep=0, fill=black, minimum width=\dotsize
	},
	pack size/.store in=\psize,
	penetration/.store in=\penetration,
  spacing/.store in=\spacing,
  link size/.store in=\lsize,
  pack color/.store in=\pcolor,
 	pack inside color/.store in=\picolor,
  pack inside color=blue!20,
 	pack outside color/.store in=\pocolor,
  pack outside color=blue!50!black,
 	surround sep/.store in=\ssep,
 	link/.style={
  	circle, 
  	draw=black, 
  	fill=black,
  	inner sep=0pt, 
 		minimum size=\lsize
 	},
  pack/.style={
 		circle, 
 		draw = \pocolor, 
  	fill = \picolor,
  	minimum size = \psize
  },
  func/.style={
  	pack,
		rectangle,
		rounded corners=.5*\psize,
		inner ysep=.125*\psize,
		minimum width=1.125*\psize,
		inner xsep=.25*\psize,
  },
  funcr/.style={
    func,
    rectangle round north west=false, 
		rectangle round south west=false,
  },
  funcl/.style={
    func,
		rectangle round north east=false, 
		rectangle round south east=false,
  },
  funcu/.style={
    func,
		rectangle round south east=false, 
		rectangle round south west=false,
  },
  funcd/.style={
    func,
		rectangle round north east=false, 
		rectangle round north west=false,
  },
  oshell/.style={
	  kite, 
	  inner sep=0pt, 
	  draw = \pocolor, 
  	fill = \picolor,
		inner ysep=.125*\psize,
		minimum width=.5*\psize,
		inner xsep=.075*\psize,
	},  
  oshellr/.style={
		oshell,
	  shape border rotate=90, 
	  },
  oshelll/.style={
		oshell,
	  shape border rotate=270, 
	  },
  oshellu/.style={
		oshell,
	  shape border rotate=180, 
	  },
  oshelld/.style={
		oshell,
	  shape border rotate=0, 
	  },
  outer pack/.style={
 		ellipse, 
 		draw,
  	inner sep=\ssep,
  	color=gray,
 	},
  intermediate pack/.style={
 		ellipse,
 		dashed, 
  	draw,
  	inner sep=\ssep,
 		color=\pocolor,
 	},
 }
 
\tikzset{light gray nodes/.style={every node/.style={fill=gray!40}}}

\tikzset{
	oriented WD/.style={
		every to/.style={out=0,in=180,draw},
    label/.style={
    	font=\everymath\expandafter{\the\everymath\scriptstyle},
      inner sep=0pt,
      node distance=2pt and -2pt},
    semithick,
    node distance=1 and 1,
    decoration={markings, mark=at position \stringdecpos with \stringdec},
    ar/.style={postaction={decorate}},
    execute at begin picture={\tikzset{
    	x=\bbx, y=\bby,
      every fit/.style={inner xsep=\bbx, inner ysep=\bby}}}
    },
    string decoration/.store in=\stringdec,
    string decoration={\arrow{stealth};},
    string decoration pos/.store in=\stringdecpos,
    string decoration pos=.7,
    bbx/.store in=\bbx,
    bbx = 1.5cm,
    bby/.store in=\bby,
    bby = 1.5ex,
    bb port sep/.store in=\bbportsep,
    bb port sep=1.5,
    bb port length/.store in=\bbportlen,
    bb port length=4pt,
    bb penetrate/.store in=\bbpenetrate,
    bb penetrate=0,
    bb min width/.store in=\bbminwidth,
    bb min width=1cm,
    bb rounded corners/.store in=\bbcorners,
    bb rounded corners=2pt,
    bb spider/.style={
    	bb port sep=1, bb port length=10pt, bbx=.4cm, bb min width=.4cm, bby=.8ex},
    bb small/.style={
    	bb port sep=1, bb port length=2.5pt, bbx=.4cm, bb min width=.4cm, bby=.7ex},
		bb medium/.style={
			bb port sep=1, bb port length=2.5pt, bbx=.4cm, bb min width=.4cm, bby=.9ex},
    bb/.code 2 args={
    	\pgfmathsetlengthmacro{\bbheight}{\bbportsep * (max(#1,#2)+1) * \bby}
      \pgfkeysalso{draw,minimum height=\bbheight,minimum
       width=\bbminwidth,outer sep=0pt,
         rounded corners=\bbcorners,thick,
         prefix after command={\pgfextra{\let\fixname\tikzlastnode}},
         append after command={\pgfextra{\draw
            \ifnum #1=0{} \else foreach \i in {1,...,#1} {
            	($(\fixname.north west)!{(2*\i-1)/(2*#1)}!(\fixname.south west)$) +(-\bbportlen,0) coordinate (\fixname_in\i) -- +(\bbpenetrate,0) coordinate (\fixname_in\i')}\fi 
            \ifnum #2=0{} \else foreach \i in {1,...,#2} {
            	($(\fixname.north east)!{(2*\i-1)/(2*#2)}!(\fixname.south east)$) +(-
\bbpenetrate,0) coordinate (\fixname_out\i') -- +(\bbportlen,0) coordinate (\fixname_out\i)}\fi;
           }}}
		},
		bb name/.style={
    	append after command={
				\pgfextra{\node[anchor=north] at (\fixname.north) {#1};}
			}
		},
  }

\tikzset{
	unoriented WD/.style={
  	every to/.style={draw},
  	shorten <=-\penetration, shorten >=-\penetration,
  	label distance=-2pt,
  	thick,
  	node distance=\spacing,
  	execute at begin picture={\tikzset{
  		x=\spacing, y=\spacing, circuit logic US, tiny circuit symbols}
		}
  },
  pack size/.store in=\psize,
  pack size = 12pt,
	penetration/.store in=\penetration,
	penetration = 0pt,
  spacing/.store in=\spacing,
  spacing = 8pt,
  link size/.store in=\lsize,
  link size = 2pt,
  pack color/.store in=\pcolor,
  pack color = violet,
 	pack inside color/.store in=\picolor,
  pack inside color=violet!20,
 	pack outside color/.store in=\pocolor,
  pack outside color=violet!50!black,
 	surround sep/.store in=\ssep,
  surround sep=8pt,
 	link/.style={
  	circle, 
  	draw=black, 
  	fill=black,
  	inner sep=0pt, 
 		minimum size=\lsize
 	},
  pack/.style={
 		circle, 
 		draw = \pocolor, 
  	fill = \picolor,
  	minimum size = \psize
  },
  func/.style={
  	pack,
		rectangle,
		rounded corners=.5*\psize,
		inner ysep=.125*\psize,
		minimum width=1.125*\psize,
		inner xsep=.25*\psize,
  },
  funcr/.style={
    func,
    rectangle round north west=false, 
		rectangle round south west=false,
  },
  funcl/.style={
    func,
		rectangle round north east=false, 
		rectangle round south east=false,
  },
  funcu/.style={
    func,
		rectangle round south east=false, 
		rectangle round south west=false,
  },
  funcd/.style={
    func,
		rectangle round north east=false, 
		rectangle round north west=false,
  },
  outer pack/.style={
 		ellipse, 
 		draw,
  	inner sep=\ssep,
  	color=gray,
 	},
  intermediate pack/.style={
 		ellipse,
 		dashed, 
  	draw,
  	inner sep=\ssep,
 		color=\pocolor,
 	},
}
  
\tikzset{
	spider diagram/.style={
		every to/.style={out=0, in=180, draw, thick},
		thick,
		dot size = 5pt,
		execute at begin picture={\tikzset{
    	x=\leglen, y=\leglen/3}}
	},
	dot size/.store in=\dotsize,
	dot fill/.store in=\dotfill,
	dot fill = black,
	leg length/.store in=\leglen,
	leg length = 15pt,
	baby/.style={dot size = 2pt, leg length = 6pt},
	young/.style={dot size = 3pt, leg length = 10pt},
	adolescent/.style={dot size = 4pt, leg length = 12pt},
	special spider/.code n args={4}{
		\pgfkeysalso{circle, draw, thick, inner sep=0, fill=\dotfill, minimum width=\dotsize,
  		prefix after command={\pgfextra{\let\fixname\tikzlastnode}},
  		append after command={\pgfextra{
  			\ifnum #1=0{} \else {\foreach \i in {1,...,#1} {
					\tikzmath{\anglei={-90*(#1+1-2*\i)/#1};}
  				\draw [thick]
						(\fixname) .. controls 
						($(\fixname.center)-(\anglei:#3/3)$) and ($(\fixname.center)-(\anglei:#3*2/3)$) .. 
						({$(\fixname.center)-(\anglei:#3*2/3)$}-|{$(\fixname.center)-(#3,0)$}) coordinate (\fixname_in\i);
  			}}\fi
  			\ifnum #2=0{} \else {\foreach \i in {1,...,#2} {
					\tikzmath{\anglei={90*(#2+1-2*\i)/#2};}
  				\draw [thick]
						(\fixname.center) .. controls 
						($(\fixname.center)+(\anglei:#4/3)$) and ($(\fixname.center)+(\anglei:#4*2/3)$) .. 
						({$(\fixname.center)+(\anglei:#4*2/3)$}-|{$(\fixname.center)+(#4,0)$}) coordinate (\fixname_out\i);
  			}}\fi
  		}}
		}
	},
	spider/.code 2 args={
		\pgfkeysalso{special spider={#1}{#2}{\leglen}{\leglen}}
	}
}

\tikzset{
	inner WD/.style={
		every to/.style={out=0, in=180, draw, thick},
		unoriented WD, 
		surround sep=2pt, 
		font=\tiny, 
		anchor=center
	}
}

\tikzset{
  function/.style={->, thin, shorten <=4pt, shorten >=4pt}
}

\tikzset{
  tick/.style={
  	postaction={
    	decorate,
      decoration={
      	markings, mark=at position 0.5 with {
					\draw[-] (0,.4ex) -- (0,-.4ex);
				}
			}
		}
	}
} 

\newcommand{\tickar}{\begin{tikzcd}[baseline=-0.5ex,cramped,sep=small,ampersand 
replacement=\&]{}\ar[r,tick]\&{}\end{tikzcd}}

\newcommand{\Tickar}[1]{\begin{tikzcd}[baseline=-0.5ex,cramped,sep=18pt, ampersand 
replacement=\&]{}\ar[r,tick, "#1"]\&{}\end{tikzcd}}

\tikzcdset{
	twocell/.style={Rightarrow, shorten <=5pt, shorten >=5pt}
}